\documentclass[a4paper,12pt]{article}
\usepackage[english]{babel}
\usepackage{tgpagella,eulervm}
\usepackage{amsmath}
\usepackage{color}
\usepackage[dvipsnames]{xcolor}
\usepackage[inline]{enumitem} 
\usepackage{authblk}
\usepackage[margin=0.8in]{geometry}
\usepackage{amsmath,amssymb,amsfonts,amsthm}%
\usepackage{tikz}
\usepackage{tikz-cd}
\usepackage{mathtools}

\usepackage{soul}
\usepackage{framed}
\usepackage[T1]{fontenc}
\usepackage{hyperref}
\hypersetup{
    colorlinks=true,
    urlcolor = {darkred},
    linkcolor = {darkred},
    citecolor = {darkred},
    linkcolor = {darkred},
    linktoc=all
}

\newcommand{\Z}{\mathbb{Z}}
\newcommand{\N}{\mathbb{N}}
\newcommand{\R}{\mathbb{R}}

\newcommand{\tot}{\mathrm{tot}}

\newcommand{\dgm}{\mathrm{dgm}}
\newcommand{\rk}{\mathrm{rk}}
\newcommand{\rank}{\mathrm{rank}}

\newcommand{\Int}{\mathrm{Int}}
\newcommand{\fpint}{\mathrm{int}}
\newcommand{\fpInt}{\fpint}

\newcommand{\res}{\chi^{\Int}}

\newcommand{\seg}{\mathrm{seg}}
\newcommand{\Seg}{\mathrm{Seg}}

\newcommand{\bfield}{k}
\newcommand{\Vect}{\mathrm{Vect}_{\bfield}}
\DeclareMathOperator{\Rep}{Rep}
\newcommand{\Repk}{\Rep}

\DeclareMathOperator{\fprep}{{fp\text{-}rep}}
\newcommand{\fprepk}{\fprep}
\DeclareMathOperator{\pfdrep}{{rep}}
\newcommand{\pfdrepk}{\pfdrep}
\DeclareMathOperator{\pfdfdrep}{{fds\text{-}rep}}
\newcommand{\pfdfdrepk}{\pfdfdrep}

\theoremstyle{remark}

\theoremstyle{definition}
\newtheorem{theorem}{Theorem}[section]
\newtheorem{convention}[theorem]{Convention}
\newtheorem{proposition}[theorem]{Proposition}
\newtheorem{corollary}[theorem]{Corollary}
\newtheorem{remark}[theorem]{Remark}

\newtheorem{definition}[theorem]{Definition}
\newtheorem{lemma}[theorem]{Lemma}

\newtheorem*{question}{Question}
\newtheorem{example}[theorem]{Example}

\newcommand{\coker}{\mathrm{coker}}
\newcommand{\dimv}{\mathrm{\underline{dim}}}
\newcommand{\dimhom}{\mathrm{dimh}}

\DeclareMathOperator{\mmod}{{mod}}
\DeclareMathOperator{\Hom}{{Hom}}

\DeclareMathOperator{\ksp}{\mathit{K}_0^{sp}}
\DeclareMathOperator{\ind}{ind}
\DeclareMathOperator{\add}{add}

\DeclareMathOperator{\convf}{Conv}

\DeclareMathOperator{\ima}{im}

\newcommand{\cC}{\mathcal{C}}
\newcommand{\cD}{\mathcal{D}}

\newcommand{\smat}[1]{ \left[\begin{smallmatrix} #1 \end{smallmatrix}\right] }

\usepackage{xcolor}%
\definecolor{darkblue}{rgb}{0.0, 0.0, 0.8}
\definecolor{darkred}{rgb}{0.8, 0.0, 0.0}
\definecolor{darkgreen}{rgb}{0.0, 0.65, 0.0}

\newcommand{\replace}[2]
{#2}
\newcommand{\repl}[2]
{#2}

\title{Barcoding Invariants and Their Comparison
}

\author[1]{Emerson G. Escolar}
\author[2]{Woojin Kim}

\affil[1]{Graduate School of Human Development and Environment,
  Kobe University,  
    {3-11 Tsurukabuto, Nada-ku},
    {Kobe},
    {657-8501},
    {Hyogo},
    {Japan}\thanks{\texttt{e.g.escolar@people.kobe-u.ac.jp} 
}}

\affil[2]{Department of Mathematical Sciences, KAIST, 
Daejeon, South Korea\thanks{\texttt{woojin.kim@kaist.ac.kr} 
}}

\begin{document}

\maketitle

\begin{abstract}
 The persistence barcode, which can be obtained from the interval decomposition of a persistence module, plays a pivotal role in applications of persistent homology.
 For \emph{multi-}parameter persistent homology, which lacks a complete discrete invariant, and where persistence modules are no longer always interval decomposable, many alternative invariants have been proposed.
 Many of these invariants are akin to persistence barcodes, in that they assign (signed) multisets of intervals.
 Furthermore, to any interval decomposable module, those invariants assign the multiset of intervals that correspond to its summands.
 Naturally, identifying the relationships among invariants of this type, or ordering them by their discriminating power, is a fundamental question.
 To address this,  we formalize the notion of \emph{barcoding invariants} and compare them by comparing their kernels, which are taken as a measure of their (in-)discriminating power.
 We show that any two different barcoding invariants $f$ and $g$ with the same basis are incomparable; i.e.\ one cannot be strictly finer than the other.
 Furthermore, we identify what we call a \emph{transfer isomorphism} between the kernels of $f$ and $g$, implying that, given any pair of persistence modules that are not distinguishable via $f$ but are via $g$, one can generate another pair of persistence modules that are so via $f$, but not via $g$.
 One implication of the existence of the transfer isomorphism is that introducing a new barcoding invariant does not add any value in terms of its generic discriminating power, even if it is distinct from the existing barcoding invariants.
 Another implication is a novel characterization of the generalized persistence diagram without involving M\"obius inversion.
 Along the way, we generalize several recent results on the discriminative power of invariants for poset representations within our unified framework.
\end{abstract}



\section{Introduction}
\label{sec:intro}

Persistent homology is a central concept in Topological Data Analysis (TDA), with many applications to date \cite{cang2015topological,carlsson2021topological,chan2013topology,dabaghian2012topological,PhysRevC.106.064912,lee2018high,stolz2023relational,hiraoka2016hierarchical}. One of the most common methods for obtaining persistent homology starts with a dataset $X$, a finite set of points in Euclidean space, and then we construct a nested sequence of simplicial complexes on $X$ that captures the structure of $X$. Applying the homology functor with field coefficients to this nested family gives us persistent homology, or more specifically, a \emph{persistence module}, i.e. a representation of a totally ordered poset.
In general, a \emph{representation of a poset} is a
functor from the poset to the category of vector spaces and linear maps over a field.

Multi-parameter persistent homology or the homology of filtered topological spaces over posets, extends the concept of persistent homology \cite{botnan2022introduction,bubenik2015metrics, carlsson2009theory}. This generalization arises when considering multiple aspects or features of the dataset $X$.
Examples include aspects (properties) such as the density of points in $X$ (in order to differentiate the role of outliers in $X$ from that of the other points in $X$ when constructing a simplicial filtration), time (when $X$ is time-varying \cite{kim2021spatiotemporal}), or domain-specific features—such as ionization energy (when $X$ stands for an atomic configuration \cite{demir2023se}). From such considerations, we obtain a representation of a poset, often a product of totally ordered sets. A representation of a product of totally ordered sets is often referred to as a \emph{multi-parameter persistence module} in TDA.

Whereas persistence modules admit \emph{persistence barcodes} as their complete discrete invariants, multi-parameter persistence modules do not admit such an invariant \cite{carlsson2009theory}.
From the perspective of representation theory, representations of a poset $P$ is of \emph{wild type} unless $P$ is one of the exceptional posets \cite{aoki2023summand,bauer2020cotorsion,derksen2005quiver}. This motivates researchers in the TDA community to investigate a proxy for persistence barcodes for representations of a general poset $P$, which are not necessarily complete, but potentially useful in practical applications of TDA.

In this effort, numerous invariants have been proposed for multi-parameter persistence modules or general poset representations. Examples include 
the rank invariant \cite{carlsson2009theory}, 
the fibered barcode \cite{lesnick2015interactive},
the Hilbert functions (a.k.a. dimension vectors) and graded Betti numbers \cite{lesnick2015interactive,oudot2024stability},
the generalized rank invariant and its M\"obius inversion \cite{kinser2008rank,kim2021generalized,kim2023persistence},
the elder-rule-staircode \cite{cai2021elder},
the zigzag-indexed-barcode \cite{clause2022discriminating},
the meta-diagram \cite{clause2023meta},
the birth-death function and its M\"obius inversion 
\cite{mccleary2022edit},
the multirank invariant \cite{thomas2019invariants},
compressed multiplicities and interval replacement \cite{asashiba2019approximation,asashiba2024interval},
connected persistence diagrams \cite{hiraoka2023refinement},
invariants using relative resolutions \cite{amiot2024invariants,botnan2021signed,blanchette2024homological,blanchette_exact_2023,asashiba2023approximation,aoki2023summand, chacholski2024koszul}, the fringe presentations \cite{miller2020homological,lenzen2024computing}, 
the graphcode\footnote{To be precise, the graphcode is not an invariant of a multi-parameter persistence module, but a feature that can be extracted from a simplicial bifiltration.} \cite{russoldgraphcode}, 
the Grassmannian persistence diagrams \cite{gulen2023orthogonal},
and
the skyscraper invariant which is based on the Harder-Narasimhan types of quiver representations \cite{fersztand2023harder, fersztand2024harder}.

Several of these invariants are akin to persistence barcodes,
in that they assign (possibly signed) multisets of intervals.
Furthermore, those invariants assign any interval decomposable module to its corresponding multiset of intervals.
Naturally, identifying the relationships among those of invariants or ordering them by their discriminating power is a fundamental question. We take a systematic approach to address this question.
We summarize our contributions in Items \ref{item:contribution-barcoding-invariant}-\ref{item:contribution-generalizations} below.

\paragraph{Contributions}
Let $P$ be a poset.
\begin{enumerate}
\item
  Building on the concept of \emph{invariant} from \cite{blanchette2024homological},
  we formalize the notion of \emph{barcoding invariants}.
  Broadly speaking, given a fixed set of indecomposable representations of $P$,
  a barcoding invariant $f$ is defined as a $\Z$-linear map 
  sending each representation of $P$
  to a $\mathbb{Z}$-linear combination of representations in the fixed set,
  so that map is the identity on this fixed set.   
  We call the fixed set the \emph{basis} for $f$;
  see Definitions \ref{definition:completefixes} and \ref{definition:barcodelike}.
  In fact, the notion of \emph{basis} has been defined for invariants in general
  (see \cite[Definition~4.11]{amiot2024invariants} and also our Remark~\ref{rem:basis}).
  The central property for ``barcoding invariants'' is that
  they are the identity on the basis.
  \label{item:contribution-barcoding-invariant}
\end{enumerate}
We remark that, in the literature, the basis is often taken to be the set of \emph{interval representations} of $P$ (cf. Definition \ref{def:interval_representation}) or its subsets \cite{asashiba2019approximation,asashiba2023approximation,asashiba2024interval,botnan2021signed,chacholski2024koszul,clause2022discriminating,dey2024computing}, mainly motivated by the goal of devising proxies for persistence barcodes in the setting of multi-parameter persistent homology.

To compare barcoding invariants, we extend the comparison framework from \cite{blanchette2024homological,blanchette_exact_2023},  where, given any two additive invariants $f$ and $g$, the relation $\ker f \subseteq \ker g$ indicates that $f$ is \emph{finer} than $g$. If neither $\ker f \subseteq \ker g$ nor $\ker g \subseteq \ker f$, we say that $f$ and $g$ are \emph{incomparable}. We say that $f$ and $g$ have \emph{equal} discriminating power if $\ker f=\ker g$.
In this case, $f$ and $g$ determine one another  (cf. Lemma \ref{lemma:invariantrelations}) and thus we also say that $f$ and $g$ are \emph{equivalent} (cf. Definition \ref{def:comparison_of_invariants}).

\begin{enumerate}[resume]
\item We show that any pair of barcoding invariants $f$ and $g$ having the same basis
  are either equivalent or incomparable; see Theorem~\ref{theorem:comparing}.
  \label{item:contribution-same-or-incomparable}
\item
We take the previous item further by 
identifying what we call a \emph{transfer isomorphism} between their kernels.    
What makes the transfer isomorphism special is that, given any pair of persistence modules that are not distinguishable via $f$ but are via $g$, one can immediately generate another pair of persistence modules that are so via $f$, but not via $g$.
See Lemma \ref{lemma:invariant_kernel}~\ref{lemma:invariantrelations:num:kerinterpret}, Theorem \ref{theorem:isokernel}, and Section~\ref{sec:applications}.\label{item:contribution-isomorphic-kernels}

One implication of Theorem~\ref{theorem:isokernel}   
is that introducing a new barcoding invariant $g$ does not add any value because, for any existing barcoding invariant $f$, there exists a bijection between $\ker f\setminus \ker g$ and $\ker g\setminus \ker f$.

\item We apply the abstract results from Items \ref{item:contribution-same-or-incomparable}-\ref{item:contribution-isomorphic-kernels} to particular invariants proposed for multi-parameter persistence modules. For example, we show that the \emph{generalized persistence diagram} (cf. Example \ref{example:gpd}) and the collection of signed intervals that naturally arises from the \emph{interval resolution} (cf. Example \ref{ex:interval_euler}) are incomparable, 
  strengthening findings from \cite{blanchette2024homological}; see Theorem \ref{thm:classifying_interval-barcoding_invariants}. Also, we identify a universal property of the generalized persistence diagram that does not involve M\"obius inversion; see Theorem \ref{thm:uniqueness}. \label{item:contribution-applications}
 
\item Our framework allows us to easily generalize results on 
  some properties shown to hold for  
  two recently introduced invariants for poset representations. This demonstrates that those results stem from the general properties of barcoding invariants, rather than their specialized constructions; see Proposition \ref{proposition:applyequivalent} and Theorem \ref{theorem:nobigger}.\label{item:contribution-generalizations}
\end{enumerate}

\paragraph{Organization}
Section \ref{sec:background} provides the preliminaries for our results. Section \ref{sec:resultscomparison} covers the content of Contributions \ref{item:contribution-barcoding-invariant}, \ref{item:contribution-same-or-incomparable}, \ref{item:contribution-isomorphic-kernels}, and \ref{item:contribution-generalizations}. Section \ref{sec:applications} discusses the content of Contribution \ref{item:contribution-applications}. Finally, Section \ref{sec:discussion} concludes the paper.


\section{Background}
\label{sec:background}

In Section \ref{sec:poset_representations}, we review basic concepts related to representations of posets. In Section \ref{sec:Krull-Schmidt}, we review the notion of Krull-Schmidt category in the context of our study. In Section \ref{sec:additive_invariants_and_their_comparisons}, we review the notion of additive invariants and their comparison framework, as introduced in the literature.  In Section \ref{subsec:invariants_examples}, we discuss several simple additive invariants. In Section \ref{subsec:Mobius_inversion}, we review the concept of M\"obius inversion and relevant additive invariants. In Section \ref{subsec:relhom_invariants}, we review basic concepts in (relative) homological algebra, and relevant additive invariants.

\subsection{Poset representations}
\label{sec:poset_representations}

Throughout this paper, $P$ denotes a poset, and $k$ denotes a fixed field. We identify $P$ with the category whose objects are the elements of $P$, and for any $x, y \in P$, the set of morphisms from $x$ to $y$ consists of the unique element $x \to y$ if $x \leq y$, and is the empty set otherwise.

A \emph{$\bfield$-representation of $P$}
(or simply \emph{representation}, when the specification of $k$ and $P$ is clear)
is defined as a functor from $P$
to the category of $\bfield$-vector spaces $\Vect$.
By $\Repk P$, we denote the category of $\bfield$-representations (i.e., the functor category), which is an additive $\bfield$-category.

The \emph{direct sum} of any two $M,N\in \Repk P$
is defined pointwisely, i.e. 
\begin{align*}
  (M\oplus N)(x)&:=M(x)\oplus N(x)\ & \text{for all } x\in P, \\
  (M\oplus N)(x\rightarrow y)&:= M(x\rightarrow y)\oplus N(x\rightarrow y)& \text{for all } x\rightarrow y\in P.
\end{align*}
 A representation $M$ is said to be
 \begin{enumerate}[label=(\roman*)]
 \item \emph{trivial} or \emph{zero} if $M(x)=0$ for all $x\in P$ (in this case, we write $M=0$), and
 \item \emph{decomposable} if $M$ is isomorphic to a direct sum of two nontrivial representations. 
 \end{enumerate}

The structurally simplest indecomposable representations are perhaps the `interval representations' that we are about to define. 
A subposet $I$ of $P$ is said to be \emph{connected} if $I$ is nonempty,
and for any pair of points $x,y\in I$,
there exists a sequence of points $x=:x_1,x_2,\ldots,x_n:=y$ in $I$ such that
either $x_i\leq x_{i+1}$ or  $x_{i+1}\leq x_{i}$ for each $i=1,\ldots,n-1$.
A full subposet $I$ of $P$ is said to be \emph{convex} if
for any elements $x \leq y \leq z$ of $P$ with $x,z \in I$, it holds that $y \in I$.
If a full subposet $I$ of $P$ is both connected and convex,
then $I$ is said to be an \emph{interval}.\footnote{Intervals have also been called \emph{spreads} \cite{blanchette2024homological}.}
Let $\Int(P)$ denote the set of all intervals in a poset $P$.
For any $x,z\in P$ with $x\leq z$, the \emph{segment} from $x$ to $z$ is
the full subposet 
$
  [x,z]:=\{y\in P:x \leq y\leq z\}.
$
Let $\Seg (P)$ denote the set of all segments in $P$.
It is clear that $\Seg(P) \subseteq \Int(P)$.

\begin{definition}
\label{def:interval_representation}
  For any $I\in \Int(P)$,
  the \emph{interval representation with support $I$} is $\bfield_I\in \Repk P$ defined
  as
  \[
    \bfield_I(x) =
    \begin{cases}
      \bfield & \text{if } x \in I \\
      0 & \text{otherwise}
    \end{cases}
    \text{ and }
    \bfield_I(x\rightarrow y) =
    \begin{cases}
      1_\bfield :\bfield \rightarrow \bfield & \text{if } x,y \in I \\
      0 & \text{otherwise}
    \end{cases}
  \]
  for any object $x$ of $P$ and any morphism $x\rightarrow y$ of $P$ (i.e.\ $x\leq y$), respectively.   
\end{definition}

\begin{proposition}[{\cite[Proposition 2.2]{botnan2018algebraic}}]
  \label{prop:interval_representation_is_indecomposable}   
  For any $I\in \Int(P)$,
  the interval representation $k_I$ of $P$ is indecomposable. 
\end{proposition}
If a given $M\in \Repk P$ is isomorphic to a direct sum of interval representations, then $M$ is said to be \emph{interval decomposable.}

For any $x\in P$, let $x^\uparrow$
denote the full subposet of all points $y\in P$ such that $x\leq y$.
Then, $x^\uparrow\in \Int(P)$.
The interval representations $k_{x^\uparrow}$, $x\in P$
are precisely indecomposable projective objects in $\Repk P$
(see for example \cite[Section~3.7]{gabriel1992representations}).

A representation $M \in \Repk P$ is said to be \emph{finitely presentable}
if 
$M$ is isomorphic to the cokernel of a morphism
between finitely generated projective representations in $\Repk P$,
i.e. finite direct sums of representations of the form $k_{x^\uparrow}$. Also, $M$ is said to be \emph{pointwise finite dimensional} if $\dim_k M(x)$ is finite for all $x\in P$.
\begin{definition}
  \label{definition:cats}
  We consider the following additive subcategories of $\Repk P$.
  \begin{enumerate}[label=(\roman*)]
  \item By $\pfdrepk P$ we denote the full subcategory of all  pointwise finite dimensional $\bfield$-representations of $P$. Objects in this category are said to be pfd $k$-representations of $P$.
  \item By $\pfdfdrepk P$ we denote the full subcategory of all
    $M \in \pfdrepk P$ that decomposes into a \textbf{f}inite \textbf{d}irect \textbf{s}um of indecomposables.
  \item By $\fprepk P$, we denote the full subcategory of all \textbf{f}initely \textbf{p}resentable
    $\bfield$-representations of $P$.
  \end{enumerate}  
\end{definition}

We will recall later that $\fprepk P$ is a subcategory of $\pfdfdrepk P$
and thus we have
$ \fprepk P \subseteq \pfdfdrepk P \subseteq \pfdrepk P $
(cf. Proposition \ref{prop:catsprops}).
While we prove the main results of this paper 
in a general setting, they will be relevant to some of the above categories.

Let us once again consider the intervals $\Int(P)$.
It is clear that for $I \in \Int(P)$,
$k_I \in \pfdfdrepk P \subseteq \pfdrepk P$.
However, in general, the representation $k_I$ may not be finitely presentable.
For example, when for $P=\R^2$ and $I=\{(x,y)\in \R^2:y>-x\}$, the $k$-representation $k_I$ of $P$ is not finitely presentable.
Let 
\[
  \fpint(P) := \{I \in \Int(P) \mid k_I \text{ is finitely presentable} \}.
\]

Similarly, even for $I\in \Seg(P)$, $k_I$ is not necessarily finitely presentable. Let
\[
  \seg(P) := \{I \in \Seg(P) \mid k_I \text{ is finitely presentable} \}.
\]

\begin{remark}\label{rem:miscellaneous}
  \begin{enumerate}[label=(\roman*)]
\item   By abuse of notation, we simply 
  write $I$ for the interval representation $k_I$.
  Accordingly, we may consider $\Int(P)$ 
  as the set of interval representations $k_I$,
  or even the set of the isomorphism classes $[k_I]$.
  A similar convention applies to $\fpint(P)$ and $\seg(P)$.\label{item:miscellaneous-abuse-of-notation}

  \item The category $\fprepk P$ with $P = \R^d$ is of particular interest in practical applications of multi-parameter persistent homology; see, e.g., \cite{botnan2022introduction, carlsson2009theory}.\footnote{In the poset $\R^d$, we have  $(x_1,\ldots,x_d)\leq (y_1,\ldots,y_d)$ if and only if $x_i\leq y_i$ for each $i=1,\ldots,d$.}
  \item    When $P$ is a finite poset, the three subcategories given in Definition~\ref{definition:cats} are identical,    
      $\Int(P) = \fpint(P) \text{, and } \Seg(P) = \seg(P).$\label{item:miscellaneous-when-P-is-finite}
  \end{enumerate}  
\end{remark}

\subsection{Krull-Schmidt categories}\label{sec:Krull-Schmidt}

In this section, we review the notion of Krull-Schmidt category,
and identify properties of the categories given in Definition \ref{definition:cats}.
Refer to \cite{mac2013categories,borceux1994handbook1} for an overview of general category theory
and \cite{borceux1994handbook2,freyd1964abelian} for additive (and abelian) categories.

An additive category $\cD$ is said to be a \emph{Krull-Schmidt category}
if every object of $\cD$
decomposes as a finite direct sum of objects of $\cD$,
each with local endomorphism ring.\footnote{A ring $R$ is \emph{local} if $1_R\neq 0_R$, and for every $x\in R$, $x$ or $1-x$ is a unit.}
In a Krull-Schmidt category,
such decompositions are unique up to permutation of summands and isomorphisms
(see for example \cite[Theorem~4.2]{krause2015krull}).
A category is said to be \emph{essentially small}
if the isomorphism classes (\emph{isoclasses}) of its objects form a set.

We clarify the hierarchy of the categories given in Definition \ref{definition:cats} and address their properties.
In what follows, $P$ stands for a poset.
\begin{proposition}
  \label{prop:catsprops}
  \begin{enumerate}[label=(\roman*)]
  \item \label{item:inclusions}  The categories
    $\fprepk P$, $\pfdfdrepk P$, $\pfdrepk P$ are essentially small,
    with the following inclusions:
    \[
      \fprepk P \subseteq \pfdfdrepk P \subseteq \pfdrepk P.
    \]

  \item \label{prop:catsprops_pfd}    
    Any $M \in \pfdrepk P$ is a direct sum of indecomposable representations
    with local endomorphism rings.
  \item \label{prop:catsprops_pfdfd}
    $\pfdfdrepk P$ is a 
    Krull-Schmidt category.
  \item \label{prop:catsprops_fp}
    $\fprepk P$ is a 
    Krull-Schmidt category,
    where $\dim_k\Hom(M,N)$ is finite for any $M, N \in \fprepk P$.
  \end{enumerate}  
\end{proposition}
\begin{proof}
  Item \ref{prop:catsprops_pfd}
  is identical to {\cite[Theorem~1.1]{botnan2020decomposition}}.
  Item \ref{prop:catsprops_pfdfd} follows from
  definition (we restrict to those objects such that the Krull-Schmidt property holds).
  Item \ref{prop:catsprops_fp} is a 
  corollary of \cite[Corollary~8.4]{gabriel1992representations}\footnote{
    Some translation is needed to adapt to the terminology of
    \cite{gabriel1992representations}, which we explain as follows.
    First, we note that    
    $\fprepk P$ is isomorphic to the category $\mmod \bfield P^{\mathrm{op}}$
    of finite-presented right modules over the $\bfield$-linearization
    of the opposite category of $P$.
    Then, for any poset, $\mathcal{A} := \bfield P^{\mathrm{op}}$ is a
    \emph{spectroid} (i.e.\ $\mathcal{A}$ is an essentially small $k$-category, with
    all endomorphism algebras local, and with distinct objects non-isomorphic, and
    $\mathcal{A}(x,y)$ is finite-dimensional for all $x,y \in \mathcal{A}$).
    Corollary~8.4 of \cite{gabriel1992representations} states that
    for $\mathcal{A}$ a spectroid, $\mmod \mathcal{A}$ is an \emph{aggregate}
    (i.e.\ an essentially small additive $k$-category, with
    each object a finite sum of objects with local endomorphism algebras,
    and $(\mmod \mathcal{A})(x,y)$ is finite-dimensional for all $x,y \in \mmod \mathcal{A}$).
  }.

  Now, we show Item \ref{item:inclusions}. We first show that $\pfdrepk P$ is essentially small, that is,
  the collection of isoclasses of representations in $\pfdrepk P$ is a set.
  For each $\{d_x\}_{x \in P} \in \prod_{x \in P} \mathbb{Z}_{\geq 0}$,
  consider the isoclasses of $M \in \pfdrepk P$
  satisfying $\dim M(x) = d_x$ for all $x\in P$. By choices of bases,
  we can consider as a representative of each isoclass the representation with $M(x) = \bfield^{d_x}$
  and linear maps $M(x \rightarrow y)$ given by left multiplication of a $d_y \times d_x$ matrix over $k$.
  The union (indexed by $\{d_x\}_{x \in P} \in \prod_{x \in P} \mathbb{Z}_{\geq 0}$) of the collections of representatives
  is clearly a set.
    
  We show the claimed inclusions.
  The inclusion $\pfdfdrepk P \subseteq \pfdrepk P$ is clear.
  To see that
  $\fprepk P \subseteq \pfdfdrepk P$, we note that
  since $\fprepk P$ is Krull-Schmidt by Item \ref{prop:catsprops_fp},
  each object $M$ in $\fprepk P$ decomposes as a finite direct sum of indecomposables.
  It remains to show that $M$ is pointwise finite dimensional.
  This follows immediately from the fact that $M$ is isomorphic
  to the cokernel of a morphism between finitely generated projective representations,
  which are pointwise finite dimensional.
\end{proof}

For an essentially small Krull-Schmidt category $\cD$
(especially $\cD = \pfdfdrepk P$ or $\cD = \fprepk P$),
we let $\ind(\cD)$ be \emph{the set of the isoclasses of indecomposable objects} in $\cD$.
For any $Q \subseteq \ind(D)$, a set of isoclasses of indecomposables of $\cD$,
the \emph{additive closure} $\add Q$ of
$Q$,
is the smallest full additive subcategory which contains $Q$ and is closed under taking direct summands.
Thus, $\add Q$ satisfies the property that it contains the zero object
and is closed under isomorphisms, direct sums, and direct summands.
The following property can be immediately checked for such subcategories.

\begin{lemma}
  \label{lemma:kscat}
  Let $\cD$ be an essentially small Krull-Schmidt category and
  $\cC$ be full subcategory of $\cD$ containing the zero object,
  closed under isomorphisms, direct sums, and direct summands.
  Then,
  \begin{enumerate}[label=(\roman*)]
  \item $\cC$ is Krull-Schmidt, and
  \item $\ind(\cC) \subseteq \ind(\cD)$.
  \end{enumerate}
\end{lemma}
\begin{framed}
  Throughout the rest of this paper, we adopt the following conventions.
  \begin{convention}  
    \label{convention:essentially_small_Krull-Schmidt}
    $\cD$ stands for an essentially small Krull-Schmidt category.
    \label{convention:subcat}
    Also, by
    $
    \cC \subseteq \cD
    $, 
    we mean that $\cC$ is a full subcategory of $\cD$ containing the zero object of $\cD$ and is
    closed under isomorphisms, direct sums, and direct summands.
  \end{convention}
\end{framed}
We will often consider the subcategory $\cC = \add Q$ of $\cD$
for some $Q \subseteq \ind(\cD)$.
Then, Lemma \ref{lemma:kscat} implies that $\cC$ is a Krull-Schmidt category
and $\ind(\cC)=Q$.

\subsection{Additive invariants and their comparisons}\label{sec:additive_invariants_and_their_comparisons}

In this section, we recall from \cite{blanchette_exact_2023} the notion of additive invariants and their comparison framework.

\begin{definition}[Split Grothendieck group]  
  The split Grothendieck group of $\cD$,
  denoted by $\ksp(\cD)$,
  is the free abelian group
  generated by isomorphism classes $[M]$ of objects in $\cD$
  modulo relations $[M_1 \oplus M_2] = [M_1] + [M_2]$ for all objects $M_1,M_2 \in \cD$.
\end{definition}

\begin{lemma}[e.g.\ {\cite[Theorem 2.3.6]{lu2013algebraic}}]
 \label{lemma:kspbasis}
  The set $\ind(\cD)$
  is a basis for $\ksp(\cD)$.
\end{lemma}

This lemma allows us to define the positive and negative parts of any element in $\ksp(\cD)$:
\begin{definition}
\label{def:pos_and_neg}
  Let $x \in \ksp(\cD)$ be nonzero. Then, 
  we can write
  \[
    x = \sum_{i=1}^n m_i [I_i]
  \]
  for some $n \in \{1,2,\hdots\}$, $[I_i] \in \ind(\cD)$, and $m_i \in \mathbb{Z} \setminus \{0\}$
  for $i \in \{1,2,\hdots, n\}$.
  We call
  \[
    X_+ := \bigoplus_{i: m_i > 0} I_i^{m_i} \in \cD
    \text{ and }
    X_- := \bigoplus_{i: m_i < 0} I_i^{-m_i} \in \cD
  \]
  the \emph{positive part} and \emph{negative part} of $x \in \ksp(\cD)$ respectively, which are unique up to isomorphism.
  When $x=0$, then the positive and negative parts of $x$ are defined to be $0$. Whether or not $x=0$,
  we have
  \[
    x = [X_+] - [X_-].
  \]
\end{definition}

For any set $S$,
the free abelian group with basis $S$ is denoted by $\mathbb{Z}^{(S)}$.
Then, each element $x \in \mathbb{Z}^{(S)}$ can be written uniquely as
$x=\sum_{s\in S} a_ss\in \mathbb{Z}^{(S)}$ with $a_s\in \Z$ for all $s\in S$
where $a_s = 0$ except for a finite number of $s\in S$.
Then, $x$
can be identified with the function
$h_x: S \rightarrow \mathbb{Z}$ with $h_x(s)=a_s$ for all $s\in S$.
This function is finitely supported.
Let $\mathbb{Z}^{S}$ be the abelian group of all functions $S \rightarrow \mathbb{Z}$.
In general, $\mathbb{Z}^{(S)}\subseteq \mathbb{Z}^S$; however, when $S$ is finite, we have 
$\mathbb{Z}^{S} = \mathbb{Z}^{(S)}$.

Let
$\cC \subseteq \cD$
(cf. Convention~\ref{convention:subcat}).
By
Lemmas~\ref{lemma:kscat} and \ref{lemma:kspbasis},
the sets
$\ind(\cC)$ and $\ind(\cD)$ are bases for $\ksp(\cC)$ and $\ksp(\cD)$, respectively.
Hence, we have:
\begin{equation}
  \label{eq:kspdiagram}
  \begin{tikzcd}
    & \mathbb{Z}^{\ind(\cC)}\ar[r,phantom, sloped, "\subseteq"]
    & \mathbb{Z}^{\ind(\cD)}
    \\
    \ksp(\cC) \ar[r,phantom, sloped, "\cong"]
    & \mathbb{Z}^{(\ind(\cC))}
    \ar[r,phantom, sloped, "\subseteq"]
    \ar[u,phantom, sloped, "\subseteq"]
    & \mathbb{Z}^{(\ind(\cD))}
    \ar[r,phantom, sloped, "\cong"]
    \ar[u,phantom, sloped, "\subseteq"]
    & \ksp(\cD).  
  \end{tikzcd}
\end{equation}

Now, let $G$ be an abelian group.
Given a map $f$ sending each object in $\cD$ to an element of $G$ such that
(i) $f$ is constant on each isomorphism class,
(ii) $f$ sends the zero object to $0_G$, and
(iii) $f(M\oplus N)=f(M)+f(N)$ for all $M,N\in \cD$,
then $f$ naturally induces a group homomorphism $\ksp(\cD) \rightarrow G$.
This motivates:

\begin{definition}[{\cite[Definition 2.13]{blanchette_exact_2023}}]
  \label{definition:additiveinvariant}    
  An \emph{additive invariant on $\cD$} is a group homomorphism 
  $f : \ksp(\cD) \rightarrow G$ for some abelian group $G$. 
\end{definition}

A natural way of defining the fineness and coarseness of an additive invariant is by measuring its discriminating power.
Since the failure to distinguish between two objects \( M, N \in \cD \) is indicated by the containment of \( [M] - [N] \) in \( \ker(f) \),
the following definition provides a natural method for comparing two additive invariants in terms of their discriminating power. 
\begin{definition}
  \label{def:comparison_of_invariants}
  Let $f$ and $g$ be additive invariants on $\cD$. We say that:
  \begin{enumerate}[label=(\roman*)]      
  \item 
    $f$ is \emph{finer} than $g$ if $\ker f \subseteq \ker g$.
    In this case, we write $f \gtrsim g$.
  \end{enumerate}
  Clearly, $\gtrsim$ is a preorder on the collection of additive invariants on $\cD$. We also say that:
 \begin{enumerate}[label=(\roman*),resume] 
\item $f$ and $g$ are \emph{incomparable} if $f \not\gtrsim g$ and $g\not\gtrsim f$.  
  \item 
    $f$ and $g$      
    \emph{have equal discriminating power} if $f \gtrsim g$ and $g \gtrsim f$, i.e. $\ker f=\ker g$.
    In this case, we write $f \sim g$, and we also say that $f$ and $g$ are \emph{equivalent}.\label{item:comparison of invariants2}
  \end{enumerate}
\end{definition}


We remark that any additive invariant $f:\ksp(\cD)\rightarrow G$
is equivalent to $f':\ksp(\cD)\rightarrow \ima f$
that is simply obtained by restricting the codomain of $f$ to its image. 

The relation $\gtrsim$ can be reformulated as follows; in fact, this reformulated version was used as the definition of $\gtrsim$ in \cite[Definition 2.4]{amiot2024invariants}.

\begin{lemma}
  \label{lemma:invariantrelations}
  Let $f$ and $g$ be any two additive invariants on $\cD$. 
  
  \begin{enumerate}[label=(\roman*)]
  \item $f \gtrsim g$ if and only if there exists a homomorphism
    $\phi : \ima f \rightarrow \ima g$ such that $g = \phi f$, i.e. the following diagram commutes.
    \[
      \begin{tikzcd}
        \ksp(\cD) \rar{f} \ar[dr]{}{g} & \ima f \dar[dashed]{\phi} \\
        & \ima g \mathrlap{.}
      \end{tikzcd}
    \]
    In words, a coarser invariant $g$ can be derived from a finer invariant $f$ through a
    homomorphism $\phi$, which can be applied for every object in $\cD$.
    \label{lemma:invariantrelations:num:equiv}

  \item $f \sim g$ if and only if there exists an isomorphism $\phi: \ima f \rightarrow \ima g$
    such that $g = \phi f$.\label{lemma:invariantrelations:num:sim}
   \end{enumerate}
\end{lemma}
The proof of Lemma~\ref{lemma:invariantrelations} is elementary, and thus we omit it.
Item \ref{lemma:invariantrelations:num:sim} implies that when two additive invariants have equal discriminating power, one completely determines the other. This justifies the two terms introduced in Definition \ref{def:comparison_of_invariants} \ref{item:comparison of invariants2}.


\begin{remark}
  \label{remark:ghomrelations}
  The proof of Lemma~\ref{lemma:invariantrelations} 
  does not actually utilize the fact that the domain is a split Grothendieck group.
  More general statements are given as follows:

  Let $f: H \rightarrow G$ and $g: H \rightarrow G'$
  be any two homomorphisms of abelian groups.
  \begin{enumerate}[label=(\roman*)]
  \item $\ker f \subseteq \ker g$ if and only if
    there exists a homomorphism $\phi : \ima f \rightarrow \ima g$ with $g = \phi f$.\label{remark:ghomrelations:num:equiv}
  \item $\ker f = \ker g$ if and only if
    there exists an isomorphism $\phi: \ima f \rightarrow \ima g$
    with $g = \phi f$.
    \label{remark:ghomrelations:num:sim}
  \end{enumerate}
\end{remark}

Next, we connect the kernels of additive invariants on $\cD$ with pairs of objects in $\cD$.

\begin{lemma}
  \label{lemma:invariant_kernel}
  Let $f$ and $g$ be any two additive invariants on $\cD$.
  \begin{enumerate}[label=(\roman*)]
  \item
    \label{lemma:invariantrelations:num:kerinterpret}
    Let $x=[X_+]-[X_-] \in \ksp(\cD)$ (cf. Definition \ref{def:pos_and_neg}).
    Then, $x \in \ker f$ if and only if
    \[
      f([X_+]) = f([X_-]).
    \]
    In words, the elements of $\ker f$ correspond to the pairs of objects in $\cD$ that $f$ cannot distinguish, and vice versa.
  \item
    \label{lemma:invariantrelations:num:interpret}
    $f \not \gtrsim g$ if and only if
    there exist $M, N \in \cD$ such that
    \begin{equation}
      \label{eq:g_can_distinguish_but_f_cant}
      f([M]) = f([N])
      \text{ but }
      g([M]) \neq g([N]).
    \end{equation}
    In words, $f$ not being finer than $g$
    means that
    there exists a pair of objects that $g$ can distinguish, but $f$ cannot.
  \end{enumerate}
\end{lemma}

\begin{proof}
\begin{enumerate}[label=(\roman*)]
  \item
    This follows immediately from the fact that
    $x = [X_+] - [X_-]$, and $f$ is a group homomorphism.

  \item Assume that $f \not \gtrsim g$,
    i.e. there exists $x\in \ker f$ with $x \not\in \ker g$.
    Let $X_+$ and $X_-$ respectively be the positive and negative parts of $x$.
    By
    Item~\ref{lemma:invariantrelations:num:kerinterpret},  
    $M:= X_+$ and $N:= X_-$
    satisfy Condition~(\ref{eq:g_can_distinguish_but_f_cant}).

    Conversely, given any $M, N \in \cD$ satisfying
    Condition~(\ref{eq:g_can_distinguish_but_f_cant}), 
    it is clear that $x := [M] - [N]$ satisfies
    $x \in \ker f$ and $x \not\in \ker g$, completing the proof.      \end{enumerate}  
\end{proof}

The following definition is useful for clarifying the discriminating power of additive invariants.
\begin{definition}
  \label{definition:complete}
  Let $\cC\subseteq \cD$ (cf. Convention~\ref{convention:essentially_small_Krull-Schmidt}).
  An additive invariant
  $f: \ksp(\cD) \rightarrow G$
  is said to be \emph{$\cC$-complete}
  if
  $
  \ker f \cap \ksp(\cC) = \{ [0] \}
  $.
\end{definition}

Let $\iota: \ksp(\cC) \subseteq \ksp(\cD)$ be the inclusion map
via the identifications made in Diagram~(\ref{eq:kspdiagram}).
Then, for any additive invariant $f: \ksp(\cD) \rightarrow G$,
we have $\ker f \cap \ksp(\cC) = \ker (f\iota)$.
Thus, 
$f$ is $\cC$-complete 
if and only if
$f$ restricted to $\ksp(\cC)$, denoted $f|_{\ksp(\cC)}$, is injective,
i.e. $f$ 
distinguishes any two distinct objects of $\cC$
(cf. Lemma~\ref{lemma:invariant_kernel}\ref{lemma:invariantrelations:num:kerinterpret}).


\subsection{First examples of additive invariants}
\label{subsec:invariants_examples}

In this section, we discuss several additive invariants. Some of these invariants will turn out to be \emph{barcoding invariants} that we will define in a later section. Readers who are familiar with these invariants may skip these examples.

\begin{example}[Trivial invariants]
  \label{ex:trivial_invariant}
  A trivial invariant is given by the identity map
  \[
    1 : \ksp(\cD) \rightarrow \ksp(\cD)
  \]
  which satisfies
  $1 \gtrsim g$ for any additive invariant $g$ on $\cD$.
  
  This invariant corresponds to an indecomposable decomposition in the following sense: Let $M\in \cD$ with $M \cong \bigoplus_{i=1}^n N_i$ with each $N_i$ indecomposable.
  Then, we have
  \[
    1([M]) = [M] = \sum_{i=1}^n [N_i]
  \]  
  and thus, under the isomorphism
  $\ksp(\cD) \cong \mathbb{Z}^{(\ind(\cD))}$ (cf. Lemma~\ref{lemma:kspbasis}),
  $1([M])$ is the map that assigns each indecomposable
  its multiplicity as a direct summand of $M$.
  Thus, we also call it the \emph{multiplicity invariant}.  
  If $P$ is a totally ordered set,
  setting $\cD = \pfdfdrep P$,
  the invariant $1$ 
  corresponds to
  the   
  \emph{barcode} or \emph{persistence diagram}
\cite{abeasis1981geometry,zomorodian2005computing,cohen2007stability,crawley2015decomposition}.
    
  At the other extreme of the spectrum, there is the \emph{zero invariant}
  \[
    0 : \ksp(\cD) \rightarrow \{0\}.
  \]
  Clearly, for any additive invariant $g$ on $\cD$, we have $g \gtrsim 0$.
\end{example}

\begin{example}
  \label{ex:restricted_multiplicity}
  For any $Q \subseteq \ind(\cD)$,
  the natural   
  projection
  \[
    \pi^Q: \mathbb{Z}^{(\ind(\cD))}\ (\cong \ksp(\cD)) \rightarrow  \mathbb{Z}^{(Q)}
  \]
  defines an additive invariant,
  which can be thought of as the multiplicity invariant \emph{restricted} to the indecomposables in $Q$.
  For example, if $\cD=\pfdfdrep P$, and $Q$ is the set of isoclasses of
  interval representations
  of $P$, then for each $M \in \cD$,
  $\pi^Q([M]) =:\pi^\Int([M])$ 
  assigns each interval representation its multiplicity
  as a direct summand of $M$.
\end{example}

\begin{example}
  \label{ex:dimensionvector}
  The \emph{dimension vector (a.k.a. Hilbert function)} is the additive invariant
  $\dimv : \ksp(\pfdfdrep P) \rightarrow \mathbb{Z}^P$
  given by\footnote{Note that, among the three subcategories considered in Definition \ref{definition:cats}, $\pfdfdrep P$ is the largest Krull-Schmidt category (cf. Proposition \ref{prop:catsprops}).
    It is convenient to work with Krull-Schmidt categories as
    Lemma~\ref{lemma:kspbasis} provides a basis for its split Grothendieck group.}
  \[
    [M]\mapsto \dimv(M):= ( \dim(M(x)) )_{x\in P}.
  \]
\end{example}

A finer invariant than the dimension vector is the rank invariant:

\begin{example}
  \label{ex:rank_invariant} 
  Let $\leq_P$ be the partial order on the poset $P$.
  The \emph{rank invariant} \cite{carlsson2009theory} is the additive invariant
  $\rk:\ksp(\pfdfdrep P)\rightarrow \Z^{\leq_P}$
  given by
  \[
    [M]\mapsto \rk_M:=(\rank\ M(x\rightarrow y))_{(x, y)\in \leq_P}.
  \] 

\end{example}

\begin{remark}
  \label{rem:leq_and_seg} 
  Via the bijection from $\leq_P$ to $\Seg(P)$  given by $(x, y)\mapsto [x,y]$,
  we have the induced isomorphism $\Z^{\Seg(P)}\cong \Z^{\leq_P}$.
\end{remark}

A finer invariant than the rank invariant is the generalized rank invariant.
First, recall that when $P$ is connected and $M\in \pfdrep P$,
the \emph{rank} of $M$, denoted by $\rank(M)$,
is defined to be the rank of the canonical linear map from the limit of $M$ to the colimit of $M$ \cite{kinser2008rank,kim2021generalized}.
This rank, by definition, does not exceed $\min_{x\in P} \dim_k M(x)$ and thus is finite.
\begin{example}  
  \label{ex:GRI}  
  Let $Q\subseteq \Int(P)$.
  The \emph{generalized rank invariant over $Q$} 
  is the additive invariant $\rk^Q:\ksp(\pfdfdrep P)\rightarrow \Z^{Q}$ given by
  \begin{align*}
    \label{eq:GRI}
    [M]\mapsto \rk_M^Q:=\left(\rank\left(M|_I\right)\right)_{I\in Q}
  \end{align*}
  where $M|_I$ is the restriction of $M$ to $I$,
  as a full subposet of $P$ \cite[Section 3]{kim2021generalized} \cite{kim2023persistence}.
  If $Q\supseteq \Seg(P)$, by postcomposing the natural projection $\Z^{Q}\rightarrow \Z^{\Seg(P)}$
  and the isomorphism $\Z^{\Seg(P)}\cong \Z^\leq$, we retrieve the rank invariant.
  When $Q=\Int(P)$, for brevity, we write $\rk^\Int$ instead of $\rk^{\Int(P)}$.  
\end{example}

In what follows, we will see a generalization of the generalized rank invariant.

A \emph{compression system} for $P$ \cite{asashiba2019approximation,asashiba2024interval}
is a family $\xi = (\xi_I)_{I \in \Int(P)}$,
where each $\xi_I$ is a poset morphism $\xi_I : Q_I \rightarrow P$
(i.e.\ a functor, when viewing posets as categories)
from some finite connected poset $Q_I$, satisfying the following conditions.
\begin{enumerate}
\item $\xi_I$ factors through the poset inclusion $I \hookrightarrow P$, for each $I \in \Int(P)$.
\item $\xi_I(Q_I)$ contains all maximal and all minimal elements of $I$, for each $I \in \Int(P)$.
\item For each $I = [x,y] \in \Seg(P) \subseteq \Int(P)$,
  there exists a $[x', y'] \in \Seg(Q_I)$ with $\xi_I(x') = x$ and $\xi_I(y') = y$.
\end{enumerate}
Given a compression system $\xi$,
for each $I \in \Int(P)$, $\xi_I$ defines a functor
$
R_I: \Repk P \rightarrow \Repk Q_I
$
via precomposition, i.e.\ for each $M \in \Repk P$, we have
$R_I(M) = M \circ \xi_I$.
\begin{example}
  \label{ex:compression_system} 
  For any $M \in \pfdfdrep P$ and $I \in \Int(P)$,
  the \emph{compression multiplicity of $I$ in $M$ under $\xi$}
  is defined to be 
  the multiplicity 
  of $R_I(k_I)$ as a direct summand of $R_I(M)$.  
  By additivity \cite[Proposition~3.13]{asashiba2024interval}, we obtain
  the additive invariant   
  \[
    c^\xi : \ksp(\pfdfdrep P) \rightarrow \mathbb{Z}^{\Int(P)}
  \]
  which we call the \emph{compression multiplicity invariant}.\footnote{
    The compression multiplicity invariant is also called the \emph{interval rank invariant}
    \cite[Definition~4.13]{asashiba2024interval}.}
  Now, we clarify how the   
  compression multiplicity invariant
  generalizes the generalized rank invariant
  (Example~\ref{ex:GRI}).
  
  Assume that $P$ is finite and define the compression system $\xi:=\tot = (\tot_I)_{I\in\Int(P)}$ by setting, for each $I \in \Int(P)$,
  $
  Q_I := I \text{ and } \tot_I: I \hookrightarrow P 
  $
  be the poset inclusion.
  Then, the compression multiplicity $c^\tot$
  is
  equal to the generalized rank invariant over $\Int(P)$ 
  \cite[Lemma~3.1]{chambers2018persistent} 
  \cite[Remark~6.16~and~Lemma~6.17]{asashiba2024interval}.
  \begin{remark}\label{rem:essentially_covers}
    A sufficient condition for a compression system \( \xi \) to yield a compression multiplicity \( c^\xi \) identical to \( c^\tot \) is known, as a generalization of  \cite[Theorem 3.12]{dey2024computing}. More specifically,
    when $\xi_I$ \emph{essentially covers} $I$ relative to $\tot$
    for all $I \in \fpInt(P)$,
    we have $c^\xi=c^\tot=\rk$
    \cite[Definition 6.7 and Corollary~6.11]{asashiba2024interval}.
  \end{remark}
\end{example}

Next, we recall the dim-Hom invariant~\cite{blanchette2024homological}:

\begin{example} 
\label{ex:dim-hom}
  Let $Q$ be any set of isoclasses of indecomposable objects in $\fprepk P$.
  The \emph{dim-Hom invariant over $Q$} is the additive invariant
  $\dimhom_{\cD}^Q : \ksp(\cD) \rightarrow \mathbb{Z}^Q$
  given by
  \[
    \dimhom_\cD^Q([M]) = (\dim_\bfield\Hom_\cD(L,M))_{[L] \in Q},
  \] 
  where Proposition~\ref{prop:catsprops} \ref{prop:catsprops_fp} guarantees that $\dim_\bfield\Hom_\cD(L,M)\in \Z$ for each $[L]\in Q$. 
  When $Q=\Int(P)$ (cf. Remark \ref{rem:miscellaneous}~\ref{item:miscellaneous-abuse-of-notation}),
  we write $\dimhom_\cD^\Int$ instead of $\dimhom_\cD^{\Int(P)}$.
  When $\cD$ is clear, we also write
  $\dimhom^\Int$.
\end{example}


\subsection{The M\"obius inversion formula and additive invariants}
\label{subsec:Mobius_inversion}

There have been many works utilizing M\"obius inversion
in TDA, starting from \cite{patel2018generalized}.
In this section, we review the M\"obius inversion formula (Section \ref{subsec:general_construction})
and explain how one obtains an additive invariant from another additive invariant
via M\"obius inversion (Section \ref{subsubsec:Mobius_inversion_of_additive_invariants}).

\subsubsection{General construction}
\label{subsec:general_construction}
We review the notions of incidence algebra and
M\"obius inversion \cite{rota1964foundations, stanley2011enumerative} in general.
Throughout this subsection, let $Q$ denote a  \emph{locally finite} poset,
i.e. for all $p,q\in Q$ with $p\leq q$, the segment
$
  [p,q]
$
is finite
(we note that in certain examples discussed later,
the poset $Q$ will be considered as a subset of $\ind(\cD)$ with an appropriate partial order).

Fix a field $\mathbb{F}$ (which may be different from the field $\bfield$ in the previous sections).
Given any function $\alpha:\Seg(Q)\rightarrow \mathbb{F}$,
we write $\alpha(p,q)$ for $\alpha([p,q])$.
\label{nom:incidence_algebra}
The \emph{incidence algebra} $I(Q, \mathbb{F})$ of $Q$ over $\mathbb{F}$
is the $\mathbb{F}$-algebra of all functions $\Seg(Q)\rightarrow \mathbb{F}$
with the usual structure of a vector space over $\mathbb{F}$,
where multiplication is given by convolution:
\begin{equation}
  \label{eq:convolution}
  (\alpha\, \beta)(p,r):=\sum_{q\in [p,r]}\alpha(p,q)\cdot\beta(q,r).
\end{equation}
Since $Q$ is locally finite, the above sum is finite and hence $\alpha\beta$ is well-defined.
The \emph{Dirac delta function} $\delta_Q\in I(Q,\mathbb{F})$ is given by 
\begin{equation}
  \label{eq:delta_function}
  \delta_Q(p,q):=
  \begin{cases}
    1,& p=q\\
    0,& \text{otherwise,} 
  \end{cases}
\end{equation}
and serves as the two-sided multiplicative identity of $I(Q,\mathbb{F})$.

\begin{remark}[{cf.~\cite{stanley2011enumerative}}]
  \label{rem:invertibility}
  An element $\alpha\in I(Q,\mathbb{F})$ admits a multiplicative inverse
  if and only if
  $\alpha(q,q)\neq 0$ for all $q\in Q$. 
\end{remark}

Another important element of $I(Q,\mathbb{F})$ is the \emph{zeta function}:
\begin{equation}
  \label{eq:zeta_function}
  \zeta_Q(p,q) := 1 \text{ for all } [p,q] \in \Seg(P).
\end{equation}
By Remark~\ref{rem:invertibility},
the zeta function $\zeta_Q$ admits a multiplicative inverse,
which is called the \emph{M\"obius function} $\mu_{Q}\in I(Q,\mathbb{F})$. 
The M\"obius function can be computed recursively as 
\begin{equation}\label{eq:mobius}
  \mu_{Q}(p,q) =
  \begin{cases}
    1,                                     & {p=q,}\\
    -\sum\limits_{p\leq r< q}\mu_{Q}(p,r), & {p<q}.
  \end{cases}
\end{equation}

Let $\mathbb{F}^Q$ denote the vector space of all functions $Q\rightarrow \mathbb{F}$. 
\label{nom:down_p}
Also, for $q\in Q$, let
\[
  q^\downarrow:=\{p\in Q: p\leq q\},
\]
called a \emph{principal ideal}.
Let $\convf(Q, \mathbb{F}) \subseteq \mathbb{F}^Q$ be the subset
\[
  \convf(Q, \mathbb{F}) := \{ f \in \mathbb{F}^Q \mid
  \text{for every } q\in Q, 
    f(r)=0 \text{ for all but finitely many } r\in q^\downarrow\}.
\]
This is clearly a subspace of $\mathbb{F}^Q$.
Elements of $\convf(Q, \mathbb{F})$ are said to be  \emph{convolvable} (over $Q$).

Each element in $I(Q,\mathbb{F})$ acts on  $\convf(Q, \mathbb{F})$ by right multiplication:
for any $f\in \convf(Q, \mathbb{F})$ and for any $\alpha\in I(Q,\mathbb{F})$, we define:
\begin{equation}
  \label{eq:function_times_matrix}
  (f\ast\alpha)(q):=\sum_{p\leq q}f(p)\,\alpha(p,q).
\end{equation}
It can be easily checked that for a locally finite poset $Q$,
$f \ast \alpha$ belongs to $\convf(Q, \mathbb{F})$.

\begin{remark}
  \label{rem:right_multiplication}
  Let $Q$ be a locally finite poset and $\alpha\in I(Q,\mathbb{F})$.
  \begin{enumerate}[label=(\roman*)]
  \item The right multiplication map
    $\ast \alpha: \convf(Q, \mathbb{F}) \rightarrow \convf(Q, \mathbb{F})$
    given by $f\mapsto f\ast \alpha$
    is an automorphism if and only if $\alpha$ is invertible.\label{item:right_multiplication}
  \item By Remark~\ref{rem:invertibility} and the previous item,
    the right multiplication map $\ast\zeta_Q$ by the zeta function
    is an automorphism on $\convf(Q, \mathbb{F})$ with inverse $\ast\mu_{Q}.$\label{item:zeta_defines_an_automorphism}
  \item As a special case, if $q^\downarrow$ is a finite set for each $q \in Q$
    (for example if $Q$ itself is a finite poset),
    then
    every function $Q\rightarrow \mathbb{F}$ is convolvable:
    $\convf(Q, \mathbb{F}) = \mathbb{F}^Q$.\label{item:finite_principal_ideals}
  \end{enumerate}  
\end{remark}

The M\"obius inversion formula 
is a powerful tool in combinatorics with widespread applications:
\begin{theorem}[M\"obius Inversion formula~{\cite{rota1964foundations}}]\label{thm:mobius}
Let $Q$ be a locally finite poset.
For any pair of convolvable  functions $f,g:Q\rightarrow \mathbb{F}$,
\begin{equation}
  \label{eq:zeta}
  \displaystyle g(q)=\sum_{r\leq q} f(r) \text{ for all } q\in Q
\end{equation}
if and only if 
\begin{equation}
  \label{eq:zeta_inverse}
  \displaystyle f(q)=\sum_{r\leq q} g(r)\cdot \mu_{Q}(r,q) \text{ for all } q\in Q.
\end{equation}
\end{theorem}
\begin{proof}
  Equation~\eqref{eq:zeta} can be represented as $g=f\ast\zeta_Q$.
  By multiplying both sides by $\zeta_Q^{-1}=\mu_Q$ on the right,
  we have $g\ast\mu_Q=f$, which is precisely Equation~\eqref{eq:zeta_inverse}.
\end{proof}

The function $f=g\ast\mu_Q$ is referred to as the \emph{M\"obius inversion} of $g$ (over $Q$).

\begin{remark}
  \label{rem:automorphism_on_integer-valued_maps}
  In Theorem~\ref{thm:mobius},
  further assume that $\mathbb{F}$ is a field containing the ring of integers $\Z$,
  such as the rationals or the reals.
  Since the M\"obius function $\mu_Q$ is an integer-valued map (cf. Equation~\eqref{eq:mobius}),
  the automorphism $\ast \mu_Q$ on $\convf(Q, \mathbb{F})$
  described in Remark \ref{rem:right_multiplication} \ref{item:right_multiplication}
  can be restricted to the automorphism on
  $\convf(Q, \mathbb{Z})$, the abelian group of integer-valued convolvable functions on $Q$.

  Furthermore,
  similar to Remark~\ref{rem:right_multiplication} \ref{item:finite_principal_ideals},
  if every prinicipal ideal of $Q$ is finite
  (for example if $Q$ itself is a finite poset),
  then $\convf(Q, \Z) = \Z^Q$, and thus this gives an automorphism $\ast \mu_Q$ on
  $\Z^Q$.
\end{remark}

\subsubsection{M\"obius inversion of additive invariants}
\label{subsubsec:Mobius_inversion_of_additive_invariants}

 The aim of this section is to demonstrate that the M\"obius inversion of an additive invariant $f$, whenever well-defined, is equivalent to 
$f$ in the sense of Definition \ref{def:comparison_of_invariants} \ref{item:comparison of invariants2}.

For any set $Q$ of isoclasses of indecomposables in $\cD$, let 
$\cC = \add Q \subseteq \cD$ (cf. Convention \ref{convention:essentially_small_Krull-Schmidt}).
We also assume that this set $Q$ is equipped with a partial order $\leq$,
and that $Q$ is a locally finite poset under this partial order.
By Remarks~\ref{rem:right_multiplication} and \ref{rem:automorphism_on_integer-valued_maps},
the right multiplication map $\ast\mu_Q$ is an automorphism on $\convf(Q, \Z)$.
Thus, whenever an additive invariant $f:\ksp (\cD)\rightarrow \convf(Q, \Z)$ is given, 
we obtain another additive invariant 
$g: \ksp(\cD) \rightarrow \convf(Q, \Z)$ by defining, for each $[M] \in \ksp(\cD)$,
\[
  g([M]) := f([M]) \ast \mu_Q,
\]
which is the M\"obius inversion of $f([M])$.
This gives the commutative diagram
\[
  \begin{tikzcd}
    \ksp(\cD)\ar[dr,swap]{}{g} \rar{f} & \convf(Q, \Z) \dar{\ast \mu_Q}\\
    & \convf(Q, \Z)
  \end{tikzcd}
\]
i.e.\ $g = (\ast\mu_Q) \circ f$ where $\ast\mu_Q$ is pointwise right multiplication by $\mu_Q$.
By a slight abuse of language we also call $g$ the M\"obius inversion of $f$, where it should be noted
that the M\"obius inversion is taken ``pointwise'', i.e.\ for each $[M] \in \ksp(\cD)$, and not on $f$ itself.
By Lemma~\ref{lemma:invariantrelations}, we have $f \sim g$. 
Hence, from Theorem \ref{thm:mobius}
and Remark~\ref{rem:automorphism_on_integer-valued_maps} we obtain the following corollary.
\begin{corollary}
  \label{cor:Mobius_does_not_affect_kernel}
  \begin{enumerate}[label=(\roman*)]
 \item  Let $Q$ be a set of isoclasses of indecomposables in $\cD$.
  Let $Q$ be equipped with a partial order $\leq$
  so that $(Q, \leq)$ is a locally finite poset. Then,
  any additive invariant
  $f: \ksp(\cD) \rightarrow \convf(Q, \Z)$
  is equivalent to its M\"obius inversion   
  $(\ast\mu_Q)\circ f: \ksp(\cD) \rightarrow \convf(Q, \Z)$. \label{item:Mobius_does_not_affect_kernel1}

 \item  Additionally, if every principal ideal of $Q$ is finite, then
  any additive invariant
  $f: \ksp(\cD) \rightarrow \Z^Q$
  is equivalent to its M\"obius inversion   
  $(\ast\mu_Q)\circ f: \ksp(\cD) \rightarrow \Z^Q$. \label{item:Mobius_does_not_affect_kernel2}
\end{enumerate}  
\end{corollary}

Examples of additive invariants that are involved with M\"obius inversion follow.
Let $P$ be any poset.

\begin{example}
  Consider the set $\Seg(P)$ of segments in $P$.
  Via the bijection from $\leq$ to $\Seg(P)$ given by $(x, y)\mapsto [x,y]$, 
  the rank invariant given in Example \ref{ex:rank_invariant} can be viewed as the additive invariant
  $\rk:\ksp(\pfdfdrepk P)\rightarrow \Z^{\Seg(P)}.$
  Assume that every principal ideal in $(\Seg(P),\supseteq)$\footnote{We clarify that this partial order on $\Seg(P)$ is defined by $I \leq J$ if and only if $I \supseteq J$.}
  is finite (for example if $P$ is a finite poset)\footnote{We remark that, even if $P$ is infinite, every principal ideal in $(\Seg(P), \supseteq)$ can be finite. An extreme example is as follows: Let $P$ be an infinite set in which no pair of points is comparable. In this case, $(\Seg(P), \supseteq)$ is isomorphic to the poset $P$, and every principal ideal of $P$ is a singleton.}.
  Then, by Corollary \ref{cor:Mobius_does_not_affect_kernel}~\ref{item:Mobius_does_not_affect_kernel2}, the M\"obius inversion $\rk \ast \mu_{(\Seg(P),\supseteq)}$ is equivalent to $\rk$. This M\"obius inversion was considered in \cite{botnan2021signed}.
\end{example}

We can generalize this example as follows.
Consider the generalized rank invariant of Example \ref{ex:GRI}.
\begin{example} 
  \label{example:gpd}  
  Let $Q \subseteq \Int(P)$ such that the induced poset $(Q, \supseteq)$ is locally finite.
  Let $\rk^Q$ be the restriction of the generalized rank invariant to $Q$. When every principal ideal of $(Q,\supseteq)$ is finite,
  by Corollary \ref{cor:Mobius_does_not_affect_kernel}~\ref{item:Mobius_does_not_affect_kernel2},
  we obtain the additive invariant 
  \[
    \rk^Q\ast \mu_{(Q,\supseteq)}:=\dgm^Q:\ksp(\pfdfdrepk P)\rightarrow \Z^{Q},
  \]
  called the \emph{generalized persistence diagram} over $Q$ \cite{clause2022discriminating}
  and is equivalent to  $\rk^Q$.
  When $Q=\Int(P)$, we use $\dgm^\Int$ instead of $\dgm^{\Int(P)}$.
\end{example}

\begin{remark}
  \label{footnote}
  In the previous example,
  even if $(Q, \supseteq)$ does not have finite principal ideals, in certain settings a generalized notion of $\dgm^Q$ can be defined and shown to be equivalent to $\rk^Q$
    via a generalization of M\"obius inversion \cite[Definition~3.1]{clause2022discriminating}.\footnote{Similar ideas can also be found in \cite{gulen2022galois}.}
  For example, consider the following two sets of assumptions: 
  \begin{enumerate}[label=(\roman*)]
  \item  $P=\R$ and $\cD:=\pfdrep P$, and $Q$ is any nonempty subset of $\Int(P)$. 
  \item 
  $P=\R^2$,  $\cD:=\fprepk P$, and $Q$ is any nonempty subset of $\Int(P)$.
 \end{enumerate}
Under either of these two sets of assumptions, (the generalized) $\dgm^Q$ is a well-defined additive invariant on $\cD$.
 Its construction, however, requires a more delicate method than what is described in this section.
 For details, we refer the reader to \cite[Section 3]{clause2022discriminating}.
\end{remark}

\begin{example}\label{example:signed_interval_mult}
  Whenever the M\"obius inversion of the compression multiplicity invariant $c^\xi$ (cf. Example \ref{ex:compression_system}) over $(\Int(P),\supseteq)$ is well-defined,
  the M\"obius inversion $c^\xi\ast \mu_{(\Int(P),\supseteq)}$ is called the \emph{signed interval multiplicity \cite{asashiba2024interval} under $\xi$}\footnote{By the isomorphisms from Diagram (\ref{eq:kspdiagram}), the signed interval multiplicity one-to-one corresponds to the \emph{interval replacement} 
    \cite[Definition 4.2]{asashiba2024interval}.}.
  In fact, in \cite{asashiba2024interval}, $P$ is assumed to be finite, and thus the M\"obius inversion is well-defined therein.
\end{example}

In the next theorem, we see the following:
(i) On the collection of interval decomposable representations, the multiplicity invariant $1$ given in Example \ref{ex:trivial_invariant} coincides with the generalized persistence diagram, and (ii) 
 the discriminating power of the generalized rank invariant $\rk^Q$ increases, when it is taken with respect to larger $Q$.

\begin{theorem}[{\cite[Theorems E and F]{clause2022discriminating} \cite[Proposition 2.4]{botnan2021signed}}]\label{thm:dgm_generalize_the_barcode_and_completeness} Let $Q$ be any set of isoclasses of interval representations of any poset $P$. Let $\cC:=\add Q \subseteq \cD:=\pfdfdrep P$. 
\begin{enumerate}[label=(\roman*)]
\item \label{item:dgm_generalize_the_barcode}
For any $M\in \cC$,
\[
  \dgm^Q([M])\footnote{    
    We remark that the definition for the \emph{function} $\dgm^Q$ in Example~\ref{example:gpd} relies on finiteness conditions on $(Q,\supseteq)$.
    However, even in cases where we do not have those finiteness conditions,
    for $[M] \in \cC$, $\dgm^Q([M]) \in \mathbb{Z}^Q$ on the left-hand side is well-defined
    as an \emph{element} of $\mathbb{Z}^Q$     
    \cite[Definition~3.1~and~Theorem C (i)]{clause2022discriminating}.}
  =1([M]).
\]
\item \label{item:GRI_completeness} Assume $Q\subsetneq Q' \subseteq \ind(\cD)$ and let $\cC':=\add Q'$. 
Then, $\rk^Q$ is $\cC$-complete, but not $\cC'$-complete.
\end{enumerate}
\end{theorem}


\subsection{Relative homological algebra and additive invariants}
\label{subsec:relhom_invariants}

In this section, we recall basic terminology of relative homological algebra 
from \cite{blanchette2024homological,asashiba2023approximation}
and relevant additive invariants.
Let $\cC \subseteq \cD$.
For a morphism $f\colon C \to M$ in $\cD$ with $C \in \cC$, 
$f$ is said be a \emph{$\cC$-cover} 
or a \emph{right minimal $\cC$-approximation}
of $M$ if 
the following two conditions hold.
\begin{enumerate}[label=(\roman*)]
  \item For any morphism $f': C' \to M$ with $C' \in \cC$, the diagram
  \[
    \begin{tikzcd}
      C' \ar[dr]{}{f'} \dar[dashed]{} & \\
      C  \rar{f} & M       
    \end{tikzcd}
  \]
  can be completed to be commutative,
  i.e., there exists a morphism $C'\rightarrow C$ in $\cC$ that makes the diagram commute. 
  \label{item:precovercondition}
\item The diagram
  \[
    \begin{tikzcd}
      C \ar[dr]{}{f} \dar[dashed]{} & \\
      C  \rar{f} & M       
    \end{tikzcd}
  \]
  can only be completed to commutativity by automorphisms of $C$.
  \label{item:rightminimalcondition}
\end{enumerate}
If $f$ satisfies Item \ref{item:precovercondition} but possibly not Item \ref{item:rightminimalcondition},
then $f$ is said to be a \emph{$\cC$-precover}, or
a \emph{$\cC$-approximation} of $M$.

As before, let $\cC \subseteq \cD$ (cf.~Convention~\ref{convention:subcat}) and 
for simplicity assume that (i) $\cD$ is an abelian category, and (ii) 
for any $M \in \cD$, an epimorphic $\cC$-cover of $M$ exists.
For example, these conditions are satisfied for 
$P$ a finite poset, $\cD = \pfdrepk P$, and any $\cC$ containing all the indecomposable projectives
(see for example \cite[Remark~4.2 and Lemma~4.3]{blanchette_exact_2023}) and \cite[Remark~3.2]{asashiba2023approximation})\footnote{Note that in general, for a non-finite poset $P$, it is possible that $\cC$-(pre)covers $f: C \rightarrow M$ may not exist.}.
Instead of the above assumptions, one can work with exact structures,
in particular the exact structure $\mathcal{F}_\cC$ induced by $\cC$ and
under the assumption of ``enough projectives''.
See \cite{blanchette_exact_2023,buhler2010exact} for more details.
Let $\Omega^0(M) := M$ and define $\Omega^i(M)$ for $i = 1,2,\hdots$ recursively as follows:
Let $f_i$ be a $\cC$-cover $f_i \colon J_i \to \Omega^{i}(M)$,
inducing a short exact sequence
\[
  0 \longrightarrow \ker f_i \overset{\iota_i}{\longrightarrow} J_i \overset{f_i}{\longrightarrow} \Omega_i(M) \longrightarrow 0, 
\]
from which we let $\Omega^{i+1}(M) := \ker f_i$.
Then, we obtain the long exact sequence
\[
  \hdots \longrightarrow J_m \overset{g_m}{\longrightarrow}
    \cdots \overset{g_2}{\longrightarrow} J_1\overset{g_1}{\longrightarrow} J_0 \overset{f_0}{\longrightarrow} M \longrightarrow 0,
\]
where
$g_i := \iota_{i-1} \circ f_i$ for each $i=1,2,\hdots$,
called a \emph{minimal $\cC$-resolution} of $M$.
If 
there exists $m\in \N$ such that $J_m\neq 0$ and $J_\ell=0$ for $\ell>m$,
then we say that the \emph{$\cC$-dimension} of $M$ is $m$.
If such $m\in \N$ does not exist, then we say that the $\cC$-dimension of $M$ is infinity.
Equivalently, the $\cC$-dimension of $M$ can be defined as
the infimum of the length of (not necessarily minimal) $\cC$-resolutions of $M$;
see, e.g., \cite[Proposition 3.9]{asashiba2023approximation}.
Finally,
the \emph{global $\cC$-dimension} (or \emph{($\cC$-)relative global dimension}) of $\cD$ 
is defined to be the supremum of the
$\cC$-dimensions of all $M \in \cD$.

For example, for a finite poset $P$, $\cD:=\pfdrepk P$,
$Q$ the set of isoclasses of all projective representations of $P$, and $\cC := \add Q$,
the
$\cC$-cover,
minimal $\cC$-resolution,
$\cC$-dimension,
and
global $\cC$-dimension
of $\cD$
correspond to the usual concepts of
projective cover,
minimal projective resolution,
projective dimension,
and global dimension of $\pfdrepk P$

In general, extra care needs to be taken for infinite posets.
Below, we discuss the particular case of $P = \mathbb{R}^d$, where we first note that
the category $\fprepk \R^d$ has been noted to be analogous to
the category of finitely generated \(\mathbb{N}^d\)-graded modules
over the \(\mathbb{N}^d\)-graded polynomial ring $k[x_1,\ldots,x_d]$
(see \cite[Theorem 1]{carlsson2009theory}, \cite[Section~2]{oudot2024stability}, \cite{lesnick2015interactive})).
For a general discussion on multigraded modules,
including the claims made in the next example, see \cite{miller2005combinatorial}.
See \cite{geist2023global,miller2020homological}
for more on the homological algebra of modules over $P = \mathbb{R}^d$.

\begin{example}
  \label{ex:multigraded_Betti}
  A representation $F\in \fprepk \R^d$ is projective\footnote{In the graded polynomial ring setting, ``free''.}
  if $F$ is isomorphic to a direct sum  $\bigoplus_{i=1}^n k_{p_i^\uparrow}$
  for some ${p_i}\in \R^d$, $i=1,\ldots,n$.
  Given any $M\in \fprepk \R^d$, 
  there exists a \emph{minimal projective resolution} of length at most $d$ 
  \[
    0 \xrightarrow{} F_d \xrightarrow{} \cdots\xrightarrow{} F_1\xrightarrow{} F_0\xrightarrow{} M \xrightarrow{} 0,
  \] 
  which is unique up to isomorphism.
  For $i=0,\ldots,d$, the \emph{$i$-th graded Betti number} $\beta_i(M)$
  is the map $\R^d\rightarrow \Z$ that sends $p\in \R^d$
  to the number of indecomposable summands of $F_i$ that are isomorphic to $k_{p^\uparrow}$.
  As $\beta_i$ is additive, for $\cD=\fprepk \R^d$,
  we obtain the corresponding additive invariant $\beta_i:\ksp(\cD)\rightarrow \Z^{(\R^d)}$.  
\end{example}

Given any poset $P$, let $P\cup\{\infty\}$ be the extension of $P$ with $a<\infty$ for all $a\in P$. For any $a,b\in P\cup\{\infty\}$ with $a<b$, let $\langle a,b \langle:=\{p\in P: a\leq p\not\geq  b\}$, called a \emph{hook} in $P$. Note that, when $b=\infty$, we have $\langle a,b \langle= a^\uparrow$. Let $\mathrm{Hook}(P)$ be the set of all hooks in $P$, which is a subset of $\Int(P)$.
\begin{example}[Rank-exact resolutions {\cite{botnan2021signed}}] 
Let $P$ be any upper-semilattice.\footnote{$P$ is said to be an \emph{upper semi-lattice} if, for every pair of points in $P$, their join exists in $P$.} Given any $M\in \fprepk P =: \cD$, 
there exists a minimal resolution
of finite length
\[0 \xrightarrow{} H_n \xrightarrow{} \cdots\xrightarrow{} H_1\xrightarrow{} H_0 \xrightarrow{} M \xrightarrow{} 0\] 
such that each $H_i$ is isomorphic to a finite direct sum $\bigoplus_{j=1}^n k_{\langle a_j,b_j\langle}$. This resolution is called a \emph{minimal rank-exact resolution}, which is unique up to isomorphism. Similar to Example \ref{ex:multigraded_Betti}, for $i=0,1,\ldots$, we obtain the \emph{$i$-th rank-exact-Betti numbers} 
$\beta_i^{\rk}:\ksp(\cD)\rightarrow \Z^{(\mathrm{Hook}(P))}$, which are additive invariants on $\cD$.
The alternating sum $\sum_i (-1)^i\beta_i^{\rk}:\ksp(\cD)\rightarrow \Z^{(\mathrm{Hook}(P))}$ is also an additive invariant, which is called the \emph{minimal rank decomposition  using hooks}. As the name indicates, given the above exact sequence, the rank invariant $\rk_M$ (cf. Example~\ref{ex:rank_invariant}) coincides with $\sum_{i}(-1)^i\rk_{H_i}$.
\end{example}

\begin{framed}
In the rest of this section, we assume that $P$ is a finite poset
(therefore, $\fpint(P)=\Int(P)$ and $\fprepk P = \pfdfdrep P=\pfdrep P $). Also, we assume that $\cC := \add Q$
for some $Q \subseteq \ind(\pfdrepk P)$
such that
(1) $Q$ is finite, and 
(2) $Q$ contains all the isoclasses of indecomposable projective representations.
\end{framed}
In particular, Assumption (2) implies that every $M \in \pfdrepk P$ admits a surjective $\cC$-cover.

For example, let 
$\cC = \add Q$ where $Q$ is the set of isoclasses of all 
interval representations in $P$.
In this setting, we call a minimal $\cC$-resolution a \emph{minimal interval resolution}, 
the global $\cC$-dimension of $\pfdrepk P$ the
\emph{global interval dimension} 
of $P$, and a $\cC$-cover an \emph{interval cover.}

\begin{proposition}[{\cite[Proposition 4.5]{asashiba2023approximation}}]
  \label{prop:finite_interval_dimension}  
  The global interval dimension of a finite poset is finite.  
\end{proposition}

Proposition~\ref{prop:finite_interval_dimension} implies that for any finite poset $P$,
any $M\in \pfdrepk P$ admits a minimal interval resolution of finite length $\ell$
(where $\ell$ is at most the global interval dimension of $P$):
\begin{equation}
  \label{eq:intervalresolution}
  0 \xrightarrow{} I_\ell \xrightarrow{} \cdots\xrightarrow{} I_1\xrightarrow{} I_0\xrightarrow{} M \xrightarrow{} 0,
\end{equation}
in which each $I_i$ is a direct sum of interval representations.
The integer $\ell$ is called the \emph{interval dimension} of $M$.

\begin{example}\label{ex:interval_euler}  For each $i\in \N\cup\{0\}$,
  the \emph{$i$-th interval Betti number} of $M$, denoted $\beta_i^{\Int}(M)$,
  is the map $\Int(P) \rightarrow \Z$ sending each $J\in \Int(P)$
  to the number of summands of $I_i$ that are isomorphic to $J$.
  We obtain the corresponding additive invariant
  \[
    \beta_i^\Int:\ksp(\pfdrepk P)\rightarrow \Z^{\Int(P)}. 
  \]
  Furthermore, by taking the alternating sum,
  we define the \emph{interval Euler characteristic}
  \[
    \res := \sum_{i=0}^{\infty} (-1)^i \beta_i^\Int : \ksp(\pfdrepk P)\rightarrow \Z^{\Int(P)}.
  \]
  By Proposition \ref{prop:finite_interval_dimension},
  the above infinite sum includes only finitely many nonzero terms.
\end{example}

The interval Euler characteristic of $M$ can also be viewed
as an element of the \emph{relative Grothendieck group of the incidence algebra of $P$, relative to interval representations} (see \cite{blanchette2024homological}).

This perspective, along with results presented in \cite{asashiba2023approximation}, leads to:
\begin{proposition}
  \label{proposition:dimhomandres}
  Let $\cD:=\pfdrep P$.
  Then, the dim-hom invariant $\dimhom_\cD^\Int$ is equivalent to the interval Euler characteristic $\res$.
\end{proposition}
Although this proposition is already given in \cite[Remark 6.6]{blanchette2024homological}
and \cite{asashiba2023approximation}, we provide a proof below.
\begin{proof}
  Let $\cC:=\add \Int(P)$.
  Proposition \ref{prop:finite_interval_dimension} together with \cite[Proposition 4.9]{blanchette2024homological}
  guarantees that the interval Euler characteristic $\res$ is
  the canonical quotient map from $\ksp(\cD)$ to $\Z^{\Int(P)}$,
  with $\Z^{\Int(P)}$ isomorphic to the Grothendieck group  relative to $\Int(P)$
  (see \cite[Definition 4.7]{blanchette2024homological} for the precise definition of the relative Grothendieck group).
  Now, \cite[Theorem 4.22]{blanchette2024homological} directly implies that  $\dimhom_\cD^\Int$ and $\res$ are equivalent.  
\end{proof}

Not only the set $\Int(P)$ of intervals, for $\cC := \add Q$ with
  $Q \subseteq \ind(\pfdrepk P)$ satisfying Assumptions (1) and (2) in the box above, 
  the \emph{$i$-th $Q$ Betti number} $\beta_i^Q$ can be similarly defined.
  If the $\cC$-dimension of each $M \in \cD$ is finite, the alternating sum
  $\chi^Q := \sum_{i=0}^\infty (-1)^i \beta_i^Q$ can be defined.
  An invariant equivalent to $\chi^Q$ is called a \emph{homological invariant relative to $Q$}
  \cite[Definition 4.12]{blanchette2024homological}.
  Theorem 4.22 of \cite{blanchette2024homological} states that if the global $\cC$-dimension is finite, then
  $\dimhom_\cD^Q$ and $\chi^Q$ are equivalent.


\section{Barcoding invariants and their comparison}
\label{sec:resultscomparison}

The critical property of the
generalized persistence diagram or interval replacement
(Examples \ref{example:gpd} and \ref{example:signed_interval_mult})
is that they serve as invariants of poset representations in terms of interval representations.
This naturally generalizes the notion of a barcode for one-parameter persistence modules. 
In this section, we abstract and generalize the concept of the generalized persistence diagram or interval replacement,
leading to the notion of a \emph{barcoding invariant}.
The key idea is that, given \( \cC \subseteq \cD \) (cf. Convention \ref{convention:essentially_small_Krull-Schmidt}),
a \( \cC \)-barcoding invariant for objects in \( \cD \) is a homomorphism \( f: \ksp(\cD) \rightarrow \ksp(\cC) \) that acts as the identity 
on the subdomain \( \ksp(\cC)\subseteq \ksp(\cD) \) (cf. Diagram (\ref{eq:kspdiagram})).
In other words, through $f$, any object in $\cD$ is \emph{described} or \emph{approximated} by a
formal $\Z$-linear combination\footnote{Note that by definition, a linear combination is finite.} of objects in $\cC$, while keeping the objects in $\cC$ unchanged.

The goals of this section are to reveal the structure
of the collection of $\cC$-barcoding invariants
under the partial order $\gtrsim$ from Definition \ref{def:comparison_of_invariants}, 
and to compare the discriminating power of $\cC$-barcoding invariants. Namely, we show that the poset of $\cC$-barcoding invariants ordered by $\gtrsim$ does not contain any pair of invariants where one strictly refines the other (Theorem \ref{theorem:comparing} and Corollary \ref{corollary:comparingbarcodelike}), and in fact,
all the $\cC$-barcoding invariants have 
isomorphic kernels via what we call a transfer isomorphism
(Theorem \ref{theorem:isokernel}).
In addition, we generalize some prior results on the discriminative power of invariants for
poset representations within our unified framework
(Proposition \ref{proposition:applyequivalent} and Theorem \ref{theorem:nobigger}).

\subsection{Barcoding invariants: definitions and examples}

\begin{definition}
  \label{definition:completefixes}  
  Let $\cC \subseteq \cD$ (cf.~Convention~\ref{convention:subcat})
  and $G$ be an abelian group.
  \begin{enumerate}[label=(\roman*)]
  \item    
    An additive invariant 
    $f: \ksp(\cD) \rightarrow \ksp(\cC)$    
    is said to be \emph{$\cC$-barcoding} if $f(c) = c$
    for any $c \in \ksp(\cC)$\footnote{In this case,     
    each element of $\ksp(\cC)$ is a fixed point of $f$. We simply say that $f$ fixes elements of $\ksp(\cC)$.}. 
    \label{definition:numfixes} 
  \item An additive invariant 
    $f: \ksp(\cD) \rightarrow G $    
    is said to be \emph{$\cC$-barcoding-equivalent}
    if $f$ is equivalent to a $\cC$-barcoding invariant.
  \end{enumerate}
  We refer to a $\cC$-barcoding(-equivalent) invariant simply as
  a \emph{barcoding(-equivalent) invariant}
  whenever the subcategory $\cC \subseteq \cD$ is clear from context.
\end{definition}
We remark that given any $\cC$-barcoding-equivalent invariant $f$ on $\cD$,
Lemma~\ref{lemma:invariantrelations}\ref{lemma:invariantrelations:num:sim} implies
that $\ima f \cong \ksp(\cC)$.
Furthermore, for such an invariant $f$, Lemma~\ref{lemma:kspbasis} shows that
$\ind(\cC)$ is a basis for $\ksp(\cC)$ and thus for $\ima f$ under the previous isomorphism.
We abuse the language and simply say that $\ind(C)$ is a \emph{basis} for the
$\cC$-barcoding-equivalent invariant $f$.\footnote{In fact, the notion of \emph{basis} has been defined for invariants in general
  (see \cite[Definition~4.11]{amiot2024invariants}). The central property for barcoding invariants is that
  they must be the identity on the basis. See Remark~\ref{rem:basis} for a more detailed discussion.}

The notions of the generalized persistence diagram and interval replacement can be abstracted as follows:
\begin{definition}[Interval-barcoding invariant]
  \label{definition:barcodelike}    
  Let $\cD:=\pfdfdrep P$ (resp. $\cD := \fprepk P$)
  and $Q$ be the set of isoclasses of all interval representations in $\cD$, i.e.\ $\Int(P)$ (resp. $\fpint(P)$) respectively.
  Let $\cC := \add Q \subseteq \cD$.
  A $\cC$-barcoding invariant 
  $\ksp(\cD)\rightarrow \ksp(\cC)$   
  is said to be \emph{interval-barcoding}.  
  Also, any $\cC$-barcoding-equivalent invariant $\ksp(\cD)\rightarrow G$ (where $G$ is an abelian group) is said to be \emph{interval-barcoding-equivalent}.  
\end{definition}

We provide a list of interval-barcoding or interval-barcoding-equivalent invariants: 

\begin{example}
  \label{example:compressionbasedinvariants} \label{example:resolutionbasedinvariants}
  Let $P$ be a finite poset. Then:
  \begin{enumerate}[label=(\roman*)]
  \item The interval multiplicity map $\pi^\Int$ (cf.\ Example~\ref{ex:restricted_multiplicity})
    is interval-barcoding.\label{ex:intmult_b}
  \item The generalized rank invariant $\rk$ (cf. Example \ref{ex:GRI}) is interval-barcoding-equivalent.\label{ex:gri_be}
  \item For \emph{any} compression system  $\xi$,
    the compression multiplicity invariant $c^\xi$ (cf.\ Example~\ref{ex:compression_system}) is interval-barcoding-equivalent.
    \label{ex:cm_be}

  \item The generalized persistence diagram (cf. Example \ref{example:gpd}) is interval-barcoding.
    \label{ex:gpd_b}
  \end{enumerate}
  For the invariants defined via relative homological algebra, we have:
  \begin{enumerate}[label=(\roman*),resume]
  \item The $0$th interval Betti number $\beta_0^\Int$ (cf. Example \ref{ex:interval_euler}) is interval-barcoding.
    \label{ex:0betti_b}
  \item The interval Euler characteristic $\res$ (cf. Example \ref{ex:interval_euler}) is interval-barcoding.
    \label{ex:euler_b}
  \item The dim-hom invariant $\dimhom_\cD^\Int$ (cf. Example \ref{ex:dim-hom}) is interval-barcoding-equivalent.
    \label{ex:dimhom_be}
  \end{enumerate}
  In the end of Section \ref{sec:applications},
  we discuss some interval-barcoding invariants in the setting where $P$ is infinite.
\end{example}

Item \ref{ex:intmult_b} follows by definition.
Items \ref{ex:cm_be}, \ref{ex:gri_be}, and \ref{ex:gpd_b} directly follow
from 
\cite[Corollary~3.21]{asashiba2024interval},
Corollary~\ref{cor:Mobius_does_not_affect_kernel}~\ref{item:Mobius_does_not_affect_kernel2} and
Theorem~\ref{thm:dgm_generalize_the_barcode_and_completeness}~\ref{item:dgm_generalize_the_barcode}.
Items \ref{ex:0betti_b} and \ref{ex:euler_b}
follow from the fact that for
an interval-decomposable $M$, its minimal interval resolution is given by:
\[
  0 \rightarrow M \rightarrow M \rightarrow 0.
\]
Item \ref{ex:dimhom_be} follows from the fact that
$\res \sim \dimhom_\cD^\Int$ (cf. Proposition~\ref{proposition:dimhomandres}).

Below, we explore the properties of additive invariants
satisfying one of the conditions in Definition~\ref{definition:completefixes}.
However, we note that these properties can be stated in terms of
general properties of homomorphisms between abelian groups.
To highlight this fact,
we provide the statements in general terms in addition to
the main statements in terms of additive invariants.
While these are elementary results from the point of view of abelian groups,
we found their implications for additive invariants of
persistence modules to be surprising and counter-intuitive.

For any two abelian groups $G$ and $H$, by $G\subseteq H$, we mean that $G$ is a subgroup of $H$.
%
%
We have the following:
\begin{lemma}
  \label{lemma:ghomfixescomplete}
  Let $f: H \rightarrow G'$ be a homomorphism of abelian groups,
  and let $G \subseteq H$.  
  The following are equivalent.
  \begin{enumerate}[label=(\roman*)]
  \item There exists a homomorphism
    $f': H \rightarrow G$ with $\ker f = \ker f'$, 
    such that $f'(g) = g$ for each $g \in G$.\label{item:ghofixescomplete1}
    
  \item 
    $f|_G$ is injective, and 
    $\{f(x) \mid {x \in J} \}$ is a generating set for $\ima f$
    for some (in fact, any) generating set $J \subset G$ of $G$.\label{item:ghofixescomplete2}

  \item 
    $f|_G$ is injective, and 
    $f(G) = \ima f$.\label{item:ghofixescomplete3}
  \end{enumerate}
\end{lemma}
\begin{proof}
  The implication \ref{item:ghofixescomplete2} $\implies$ \ref{item:ghofixescomplete3} is clear.
  We prove \ref{item:ghofixescomplete1}$\implies$\ref{item:ghofixescomplete2}
  and
  \ref{item:ghofixescomplete3}$\implies$\ref{item:ghofixescomplete1}.
  
  \ref{item:ghofixescomplete1}$\implies$\ref{item:ghofixescomplete2}:
  We first show that 
  $f|_G$ is injective.
  Let $f':H\rightarrow G$ be such that $\ker f = \ker f'$, and 
  $f'(g) = g$ for each $g \in G$.
  By Remark~\ref{remark:ghomrelations}\ref{remark:ghomrelations:num:sim},
  there exists an isomorphism $\phi: \ima f \rightarrow \ima f'$ such that
  $f' = \phi f$.
  Let $\iota: G \subseteq H$ be the inclusion map. Then, we have
  \[
    \ker (f|_G) = \ker (f\iota) = \ker (\phi f\iota) = \ker (f'\iota) = \{0\},
  \]
  where the last equality follows from the fact that $f'(g) = g$ for any $g \in G$
  (i.e.\ $f'|_G = f'\iota$ is the identity map). Therefore, $f|_G$ is injective.

  Let $J \subset G$ be any generating set for $G$.
  Then, for each $z \in \ima f$, we can write $\phi(z)\in \ima f' = G$ as a
  $\mathbb{Z}$-linear combination of elements from $J$.
  Since each $x\in J$ satisfies that $x=f'(x)=\phi(f(x))$,
  by applying the inverse $\phi^{-1}$ to the $\Z$-linear combination,
  we see that $z$ equals a
  $\mathbb{Z}$-linear combination of elements from
  $\{f(x) \mid x \in J \}$, as desired.

  \ref{item:ghofixescomplete3}$\implies$\ref{item:ghofixescomplete1}:
  By assumption, the restriction of $f$ to $G$ is an isomorphism $G \rightarrow \ima f$.
  Let $\phi:\ima f\rightarrow G$ be the inverse of this isomorphism.
  Then, the map $f':=\phi f$ satisfies the condition described in the statement, as desired.
\end{proof}

Lemma~\ref{lemma:ghomfixescomplete} specializes to the following,
which relates the concepts of $\cC$-barcoding-equivalent invariants and $\cC$-complete invariants (Definition~\ref{definition:complete}):
\begin{proposition}
  \label{proposition:fixescomplete}
  Let $\cC \subseteq \cD$ (cf. Convention \ref{convention:essentially_small_Krull-Schmidt}),
  and let $f$ be any additive invariant on $\cD$.
  The following are equivalent.
  \begin{enumerate}[label=(\roman*)]
  \item $f$ is $\cC$-barcoding-equivalent. \label{item:fixescomplete1}
  \item $f$ is $\cC$-complete and $\{f([c]) \}_{c \in \ind(\cC)}$
    is a generating set for $\ima f$. \label{item:fixescomplete2}
  \item $f$ is $\cC$-complete and $f(\ksp(\cC)) = \ima f$. \label{item:fixescomplete3}
  \end{enumerate}
\end{proposition}
\begin{proof}
We note that by Lemma~\ref{lemma:kspbasis},
  $\{ [c]  \}_{c \in \ind(\cC)}$
  is a generating set for $\ksp(\cC)$ (in fact, a basis).  
  Then, the statement follows from Lemma~\ref{lemma:ghomfixescomplete} with $G = \ksp(\cC)$.
\end{proof}

  \begin{remark}
    \label{rem:basis}
    We 
    remark that the existence of a basis in the sense of \cite[Definition~4.11]{amiot2024invariants}  implies barcoding-equivalence.
    In our notation, 
    \cite[Definition~4.11]{amiot2024invariants} says that a basis of  an additive invariant $f:\ksp(\cD) \rightarrow G$ is 
    a set $Q$ of isoclasses of indecomposable objects in $\cD$
    such that $f$ induces an isomorphism 
    $\begin{tikzcd}[ampersand replacement=\&]
    \langle Q \rangle \rar{f|_{\langle Q \rangle}} \& \ima f\text{.}
    \end{tikzcd}$
    Here, $\langle Q \rangle$ is the free abelian subgroup of $\ksp(\cD)$ generated by $Q$.
    By Proposition~\ref{proposition:fixescomplete},  $f$ having a basis in this sense
    is equivalent to $f$ being $\cC$-barcoding-equivalent (with $\cC = \add Q$).
  \end{remark}

The following remark will be useful later.
\begin{remark}
  \label{remark:doubleapply}  
  Let  $f: \ksp(\cD) \rightarrow \ksp(\cC)$ be $\cC$-barcoding.
  Then, for any $x \in \ksp(\cD)$,
  $f(x) \in \ksp(\cC)\subseteq \ksp(\cD)$ and thus $f(f(x)) = f(x)$ as $f$ fixes $\cC$.  
  More precisely, we have
  $
    f \iota f = f
  $
  where $\iota : \ksp(\cC) \hookrightarrow \ksp(\cD)$
  is the natural inclusion map (cf.\ Diagram~\eqref{eq:kspdiagram}).
  By abuse of notation, we simply write $f f = f$.
  More generally, given any invariant $g$ on $\cD$ with codomain $\ksp(\cC)$, it follows that
  $fg = g$.  
\end{remark}

\subsection{Generalizations of prior results} 

We generalize
\cite[Theorem~5.12]{asashiba2019approximation} and \cite[Theorem~4.14]{asashiba2024interval},
which state that their proposed invariant, i.e., the interval replacement,
preserves the (interval) rank invariant:

\begin{proposition}
  \label{proposition:applyequivalent}  
  Let $f$ be a $\cC$-barcoding invariant on $\cD$,
  and let $g$ be an additive invariant on $\cD$  
  with $f \gtrsim g$.
  Then,
  $gf = g$.
\end{proposition}
\begin{proof}
  By Lemma~\ref{lemma:invariantrelations}~\ref{lemma:invariantrelations:num:equiv},  
  there exists a   
  homomorphism $\phi: \ima f \rightarrow \ima g$ with
  $g = \phi f$. Then
  $
  gf = \phi f f = \phi f = g.
  $  
\end{proof}

By this proposition, the $\cC$-barcoding invariant $f$ preserves the 
coarser or equivalent invariant $g$,
in the sense for any $x \in \ksp(\cD)$,
the $g$-value of $x$ is the same as the $g$-value of $f(x)$.
This in particular implies that the property of preserving an equivalent invariant is one of the general properties of $\cC$-barcoding invariants, not unique to 
the interval replacement in \cite{asashiba2024interval}.

Next, we generalize Theorem~\ref{thm:dgm_generalize_the_barcode_and_completeness}~\ref{item:GRI_completeness} as follows:

\begin{theorem}[Optimality of Completeness]
  \label{theorem:nobigger}
 Let $\cC \subsetneq \cC' \subseteq \cD$
  and let  
  $f$ be a $\cC$-barcoding-equivalent invariant on $\cD$ (and thus $\cC$-complete;
  cf.\ Proposition~\ref{proposition:fixescomplete}).  
  Then, $f$ is not $\cC'$-complete (and thus not $\cC'$-barcoding-equivalent).
\end{theorem}
\begin{proof}
  Let $f:\ksp(\cD)\rightarrow G$ be $\cC$-barcoding-equivalent invariant.
  By Lemma~\ref{lemma:invariantrelations}\ref{lemma:invariantrelations:num:sim},
  there exists an isomorphism $\phi:G\rightarrow \ksp(\cC)$ such that $\phi f$ is $\cC$-barcoding and thus $\phi f$ fixes any element of $\ksp(\cC)$.
  Let $x \in \ksp(\cC') \setminus \ksp(\cC)$
  and 
  $t := x - \phi f(x) \in \ksp(\cC')$,
  where we consider $\phi f(x) \in \ksp(\cC) \subseteq \ksp(\cC')$
  as an element of $\ksp(\cC')$.
  We have:
  \[\phi f(t) = \phi f(x-\phi f(x)) = \phi f(x) - \phi f(\phi f(x)) =\phi f(x) - \phi f(x) = 0,\]
  which shows $t\in \ker \phi f = \ker f$.
  Note that $t \neq 0$ since otherwise $x = \phi f(x) \in \ksp(\cC)$, which is a contradiction.
  This shows  that $0 \neq t \in  \ker f \cap \ksp(\cC')$, i.e.\ $f$ is not $\cC'$-complete.
\end{proof}

\subsection{Barcoding invariants form a discrete poset} 

Next, we will show that the partial order $\gtrsim$ (cf. Definition \ref{def:comparison_of_invariants}),
when restricted to $\cC$-barcoding invariants,
is \emph{discrete} in the sense that
$f\gtrsim g$ if and only if $f=g$.
We also recall that $f$ and $g$ are said to be \emph{incomparable} if $f \not\gtrsim g$ and $g \not \gtrsim f$.

\begin{lemma}
  \label{lemma:ghomequal}
  Let $f,g: H\rightarrow G$ be homomorphisms of
  abelian groups with $G \subseteq H$ 
  and $f(x) = x = g(x)$ for each $x \in G$.
  If $\ker f \subseteq \ker g$, then $f=g$.
\end{lemma}
\begin{proof}
 If $\ker f \subseteq \ker g$,
 then, by Remark~\ref{remark:ghomrelations}~\ref{remark:ghomrelations:num:equiv},
 there is a group homomorphism $\phi : \ima f \rightarrow \ima g$ such that
 $g = \phi f$. 
 Also, since $f$ and $g$ fix elements of $G$, $\ima f = G = \ima g$.
 Then, for any $x \in G \subseteq H$, we have
 $
 x = g(x) = \phi(f(x)) = \phi(x),
 $
 i.e.\ $\phi$ is the identity map on $G$. This implies $g=f$.
\end{proof}

\begin{theorem}[Barcoding invariants form a discrete poset]
  \label{theorem:comparing}  
  Let $\cC \subseteq \cD$ (cf. Convention \ref{convention:essentially_small_Krull-Schmidt}),
  and let
  $f,g: \ksp(\cD) \rightarrow \ksp(\cC)$ be 
  $\cC$-barcoding.
  Then, either
  \begin{enumerate}[label=(\roman*)]
  \item $f$ and $g$ are incomparable, or
  \item $f = g$.
  \end{enumerate}
\end{theorem}
\begin{proof}
  Suppose that $f \gtrsim g$.
  Then, we have $\ker f \subseteq \ker g$ and thus $f = g$ by
  Lemma~\ref{lemma:ghomequal}.
  Similarly, the assumption $g \gtrsim f$ also implies that $f = g$.
  This completes the proof.
\end{proof}

By weakening the assumption that $f$ and $g$ are $\cC$-barcoding
to the condition that $f$ and $g$ are $\cC$-barcoding-\emph{equivalent},
we obtain:
\begin{corollary}
  \label{corollary:comparingbarcodelike}  
  Let $\cC \subseteq \cD$,
  and  suppose that
  both $f$ and $g$ are $\cC$-barcoding-equivalent invariants on $\cD$ (whose codomains may differ).
  Then, either
  \begin{enumerate}[label=(\roman*)]
  \item $f$ and $g$ are incomparable, or
  \item $f \sim g$.
  \end{enumerate}
\end{corollary}

In the previous corollary,
$f$ and $g$ are assumed to have $\ind(C)$ as a common basis.
By allowing $f$ and $g$ to have different bases,
we obtain a variation of this corollary. More specifically, 
if $g$ is barcoding-equivalent with basis containing the basis for $f$, then $f$ cannot be finer than $g$.

\begin{corollary}\label{cor:comparison}
 Let $\cC \subsetneq \cC' \subseteq \cD$, and  
 let $f$ and $g$ be
   $\cC$-barcoding-equivalent  
  and $\cC'$-barcoding-equivalent, respectively.
  Then, either
  \begin{enumerate}[label=(\roman*)]
  \item $f$ and $g$ are incomparable, or
  \item $g \gnsim f$.
  \end{enumerate}
\end{corollary}
\begin{proof}  
  Without loss of generality, suppose that
  $f$ 
  is $\cC$-barcoding
  and
  $g$ 
  is $\cC'$-barcoding
  (and thus have codomains $\ksp(\cC)$ and $\ksp(\cC')$ respectively).  
  Consider the natural projection 
  $\pi : \ksp(\cC') \rightarrow \ksp(\cC)$,
  which induces the following diagram:
  \begin{equation}\label{eq:f_and_g}
    \begin{tikzcd}
      \ksp(\cD) \rar{g} \ar[dr]{}{\pi g} \ar[dr,bend right,swap]{}{f} & \ksp(\cC') \dar{\pi} \\
      & \ksp(\cC) \mathrlap{.}
    \end{tikzcd}
  \end{equation}
  Since $\cC \subseteq \cC'$ and $g$ fixes elements of $\ksp(\cC')$,
  $\pi g$ fixes elements of $\ksp(\cC)$.  
  By Theorem~\ref{theorem:comparing},
  either $\pi g = f$, or $\pi g$ and $f$ are incomparable.

  \begin{itemize}
  \item Suppose that $\pi g = f$, which implies that
    $\ker g \subseteq \ker \pi g = \ker f$, i.e. $g \gtrsim f$.
    Also, by Theorem~\ref{theorem:nobigger},
    $f$ cannot be $\cC'$-complete and thus $g \gnsim f$. 
  \item Suppose that $\pi g$ and $f$ are incomparable, which implies
    $f \not \gtrsim \pi g$ and $\pi g \not \gtrsim f$.
    Now, $f \not \gtrsim \pi g$ implies that $\ker f \not \subseteq \ker \pi g$,
    i.e. there exists $x \in \ksp(\cD)$ with
    $f(x) = 0$ and $\pi (g(x)) \neq 0$.
    This implies that
    $g(x) \neq 0$, and thus $\ker f \not \subseteq \ker g$,
    i.e.\ $f \not \gtrsim g$.
    This implies that either $g \gnsim f$, or $f$ and $g$ are incomparable.

    On the other hand, $\pi g \not \gtrsim f$ implies that
    there exists $x \in \ksp(\cD)$ with
    $\pi(g(x)) = 0$ and $f(x) \neq 0$.
    However, we cannot conclude    
    anything interesting from this.
    In particular, this does not necessarily imply $g(x) = 0$.
  \end{itemize}
\end{proof}

\subsection{Transfer Isomorphism between kernels of barcoding invariants}
\label{subsec:transfer iso} 
Recall from Corollary \ref{corollary:comparingbarcodelike} that,
when $\cC\subseteq \cD$, the collection  
of $\cC$-barcoding-equivalent invariants on $\cD$ does not contain any pair of invariants where one strictly refines the other.
We take this further as follows.  
For each pair of $\cC$-barcoding-equivalent invariants,
we establish a concrete and special isomorphism, called a \emph{transfer isomorphism},
between their kernels.
In particular, using the transfer isomorphism, given any pair of persistence modules that are not distinguishable via an invariant $f$ but are via another invariant $g$, one can immediately generate another pair of persistence modules that are so via $f$, but not via $g$.

Given two $\cC$-barcoding invariants
$f,g: \ksp(\cD) \rightarrow \ksp(\cC)$,
suppose that there exists $x \in \ker f$ such that $x \not\in \ker g$
(and thus\ $f \not\gtrsim g$). 
Then,  since
$g(x) \in \ksp(\cC)$, we have $f(g(x)) = g(x)$ and $g(g(x)) = g(x)$
(cf.\ Remark~\ref{remark:doubleapply}).
Hence, for
$y := x - g(x)$,
we have
\begin{align*}
  f(y) &= f(x-g(x)) = f(x) - f(g(x)) = f(x) - g(x) = -g(x) \neq 0, \\
  g(y) &= g(x-g(x)) = g(x) - g(g(x)) = g(x) - g(x) = 0,
\end{align*}  
which proves $y \not\in \ker f$,  $y \in \ker g$,
and in turn
$g \not\gtrsim f$.
We have:
\begin{theorem}[Transfer isomorphism between the kernels of barcoding invariants]
  \label{theorem:isokernel}
  Let $\cC \subseteq \cD$,
  and  
  let $f$ and $g$ be $\cC$-barcoding.
  Then, the map $T$ defined by
    \begin{equation}
      \begin{array}{rccl}
        T: & \ker f & \rightarrow &  \ker g \\ 
           & x & \mapsto &  x - g(x)
      \end{array}
    \end{equation}
    is an isomorphism, and
    restricts to  a bijection $ \ker f \setminus \ker g \to \ker g \setminus \ker f$, and fixes the intersection $\ker f\cap \ker g$.
\end{theorem}
We call $T$ the \emph{transfer isomorphism} for the ordered pair $(f, g)$.
We note that $T$ is the restriction of $\mathrm{id}_{\ksp(D)} - g:\ksp(D) \rightarrow \ksp(D)$ to $\ker f$. 

\begin{proof}  
  For any $x\in \ker f$, we have
  \[
    g(T(x)) = g(x-g(x)) = g(x) - g(g(x)) = g(x) - g(x) = 0,
  \]
  which shows $T(x)\in \ker g$.
  It is also clear that $T$ is a homomorphism, whose 
  inverse is $S: \ker g \rightarrow \ker f$ defined by $S(y) := y - f(y)$.
  Indeed,
  for any $x \in \ker f$, $(S \circ T)(x) = x-g(x)-f(x)+f(g(x)) = x$ (cf.\ Remark~\ref{remark:doubleapply}),
  and symmetrically
  $T \circ S=\mathrm{id}_{\ker g}$. 
  
  In the paragraph before this theorem, we already saw that $T$ restricts to a map from $ \ker f \setminus \ker g$ to $\ker g \setminus \ker f$. Note that $S$ restricted to $\ker g \setminus \ker f$ is the inverse to this restriction, and thus $T$ restricts to a bijection $ \ker f \setminus \ker g \to \ker g \setminus \ker f$. It is easy to see that $T$ fixes the intersection $\ker f\cap \ker g$, completing the proof.
\end{proof}

\begin{remark}  
We give an alterative construction of the transfer isomorphism.
  Let $i$ be the inclusion $\ksp(\cC)\rightarrow \ksp(\cD)$ and $q:\ksp(\cD)\rightarrow \coker\ i$ be the quotient map.  
  Then, because $fi=1_{\ksp(C)}$, the following exact sequence
  \[0\to \ksp(C) \stackrel{i}{\to} \ksp(D) \stackrel{q}{\to} \coker\ i \to 0 \]
  splits, by the splitting lemma.
  This shows that the map
  \begin{align*}
(f, q) : \ksp(\cD) &\longrightarrow \ksp(\cC) \oplus \coker\ i \\
x &\longmapsto (f(x), q(x))
\end{align*}
  is an isomorphism. 
  We claim that the isomorphism $(f, q)$ restricts to the kernel $\ker f$, mapping it onto its image as follows:
  \begin{align*}
    (f, q)|_{\ker f} : \ker f &\xrightarrow{\cong} 0 \oplus \coker\ i \cong \coker\ i= \ksp(\cD) / \ksp(\cC).
  \end{align*}
  To see surjectivity, we note that for any $[x] := x + \ksp(\cC) \in \ksp(\cD) / \ksp(\cC)$,    
  $x - f(x) \in \ker f$ is the preimage of $(0, [x]) \in 0 \oplus \coker\ i$ under $(f,q)$.
  Similarly, $(g,q)$ restricted to $\ker g$ gives an isomorphism
  $\ker g \rightarrow 0 \oplus \coker\ i$.
  Hence, we have the isomorphism  
  \[
    \begin{tikzcd}[ampersand replacement=\&,column sep=6em]
      \ker f \rar{{(f,q)|_{\ker f}}} \& 0 \oplus \coker\ i  \rar{\left[(g,q)|_{\ker g}\right]^{-1}} \& \ker g. 
    \end{tikzcd}
  \]  
  This isomorphism coincides with the transfer isomorphism $T$: for any
  $x \in \ker f$,  
  \[
    \left[(g,q)|_{\ker g}\right]^{-1}((f,q)(x)) =
    \left[(g,q)|_{\ker g}\right]^{-1}((0, x+ \ksp(\cC))) =
    x - g(x) = T(x).
  \]
\end{remark}

\begin{remark}
  \label{rem:special properties of T}  
  We remark on properties of the transfer isomorphism $T$ that do not
  necessarily hold for arbitrary isomorphism $\ker f \to \ker g$.
  \begin{enumerate}[label=(\roman*)]
  \item The transfer isomorphism $T$ induces an isomorphism 
    \[
      \frac{\ker f}{\ker f \cap \ker g}\to \frac{\ker g}{\ker f \cap \ker g}.
    \]   

  \item   
    Because $T$ restricts to  a bijection $ \ker f \setminus \ker g \to \ker g \setminus \ker f$,
    this precludes the \emph{strict} containment of $\ker f$ in $\ker g$ and vice versa.     
    Hence, Theorem~\ref{theorem:isokernel} implies Corollary~\ref{corollary:comparingbarcodelike}:
    Any two $\mathcal{C}$-barcoding-equivalent invariants $f$ and $g$ are either:
    (i) incomparable, which occurs when both $\ker f \setminus \ker g$ and $\ker g \setminus \ker f$ are nonempty and in bijection with each other; or 
    (ii) equivalent ($f \sim g$), which occurs when $\ker f \setminus \ker g = \ker g \setminus \ker f = \emptyset$ (i.e., $\ker f = \ker g$).    
    \label{item:generic iso and gtrsim}

   
\end{enumerate}
\end{remark}

The transfer isomorphism  for barcoding invariants
immediately serves as a transfer isomorphism for respective equivalent invariants.  
\begin{corollary}
  \label{corollary:isokernel-equiv}
  Let $\cC \subseteq \cD$,
  and  
  let $f'$ and $g'$ be $\cC$-barcoding-equivalent.
  Then, the transfer isomorphism $T$ for the pair $(f,g)$
  where $f$ and $g$ are barcoding invariants with $f\sim f'$ and $g\sim g'$,
  is a transfer isomorphism $T: \ker f' \rightarrow \ker g'$.
\end{corollary}



As a consequence, between any two of the additive invariants $f$ and $g$
given in Example \ref{example:compressionbasedinvariants}
there exists the transfer isomorphism in each direction.
Furthermore, Lemma \ref{lemma:invariant_kernel}~\ref{lemma:invariantrelations:num:kerinterpret} implies that the transfer isomorphism
sends the pairs of $k$-representations of $P$
that are not distinguishable by $f$ to the pairs of $k$-representations of $P$
that are not distinguishable by $g$, and vice versa.

To give a concrete depiction of the previous theorem and corollary, we make the following observation.

\begin{remark}\label{rem:depiction}
  For persistence modules over a commutative grid $P$, the generalized rank invariant $\rk^\Int$ determines the bigraded Betti numbers \cite{kim2021bettis}.
  However, 
  it is not difficult to see that $\pi^\Int$ does not determine the bigraded Betti numbers. These facts might give the impression that  $\rk^\Int$ is a more delicate invariant than $\pi^\Int$.
  However, within any collection of persistence modules sharing the same generalized rank invariant, $\pi^\Int$ clearly serves as a more delicate invariant than $\rk^\Int$.
  An implication of Theorem~\ref{theorem:isokernel} is that while neither of $\rk^\Int$ nor $\pi^\Int$
  is 
  strictly finer than the other, they have isomorphic kernels via the transfer isomorphism.
\end{remark}

We emphasize that the gist of Theorem~\ref{theorem:isokernel} is not the mere existence of \emph{an} isomorphism between kernels, but the \emph{transfer} isomorphism itself.
In fact, the mere existence of an isomorphism $\ker f\cong \ker g$ can 
be guaranteed through standard arguments under mild assumptions.
However, such an existence is too weak to characterize the relation $\gtrsim$,
as detailed in Remarks~\ref{rem:genericisoistooweak1}~and~\ref{rem:generic iso is too weak} below.

\begin{remark}
  \label{rem:genericisoistooweak1}
  Let $\cC \subseteq \cD$,
  and  
  let $f$ and $g$ be $\cC$-barcoding as in Theorem~\ref{theorem:isokernel}.
  Under the extra assumption that $\ind(\cD)$ is finite,
  the following ``rank argument" proves that $\ker f$ and $\ker g$ are isomorphic.  
  
  Note that, since $f$ is surjective, and $\ksp(\cC)$ is a free abelian group (thus projective),
  the exact sequence $0\rightarrow \ker f \rightarrow \ksp(\cD) \stackrel{f}{\rightarrow} \ksp(\cC)\rightarrow 0$ splits, i.e. 
  \[\ksp(\cD) \cong \ker f \oplus \ksp(\cC),\ \mbox{and similarly,}\ 
    \ksp(\cD) \cong \ker g \oplus \ksp(\cC).\]  
  Since $\ksp(\cD)$ is a free abelian group of finite rank, so are $\ker f$ and $\ker g$. Thus, $\ker f \cong \ker g$ $\cong \Z^m$, where $m$ equals the difference between the ranks of $\ksp(\cD)$ and $\ksp(\cC)$.
  However, an arbitrary isomorphism between $\ker f$ and $\ker g$ need not satisfy the properties of a transfer isomorphism, and in particular, the existence of such an isomorphism alone does not imply Corollary~\ref{corollary:comparingbarcodelike}, unlike in the case of a transfer isomorphism (cf.~Remark~\ref{rem:special properties of T}~\ref{item:generic iso and gtrsim}).
\end{remark}

The following remark highlights
that the mere existence of some isomorphism is too coarse as a lens for comparing barcoding invariants.

\begin{remark}\label{rem:generic iso is too weak}  By extending the previous rank argument, we can even sometimes show that two barcoding invariants \emph{over different bases} have isomorphic kernels. For example, let $\cC,\cC'\subseteq \cD$ such that $\ind(\cD)\setminus \ind (\cC)$ and $\ind(\cD)\setminus \ind (\cC')$ have the same cardinality; this is always the case when $\ind(\cD)$ is infinite while both $\ind (\cC)$ and $\ind (\cC')$ are finite
  (even if $\cC\subsetneq \cC'$).
  Then, the projection maps from $\ksp(D)$ to $\ksp(C)$ and to $\ksp(C')$ are $\cC$- and $\cC'$-barcoding invariants, respectively. Since their kernels are freely generated by the bases $\ind(\cD)\setminus \ind (\cC)$ and $\ind(\cD)\setminus \ind (\cC')$, which have the same cardinality, their kernels are obviously isomorphic.
  This implies that even when $\cC \subsetneq \cC'$, the projection to $\ksp(\cC')$ is strictly finer than the projection to $\ksp(\cC)$ under the order $\gtrsim$, despite their kernels being isomorphic.
\end{remark}



\section{Applications}\label{sec:applications}

In this section, we present
(i) applications of Theorem~\ref{theorem:comparing} and Corollary~\ref{corollary:comparingbarcodelike}, and
(ii) examples illustrating how the transfer isomorphism can be used to produce pairs of representations that are distinguished by one invariant but not by another.

\subsection{Applications to interval-barcoding-equivalent invariants
  for representations of finite posets}

In this section, we classify the interval-barcoding-equivalent invariants considered in Example \ref{example:compressionbasedinvariants}, using the equivalence relation $\sim$ from Definition \ref{def:comparison_of_invariants} (cf. Theorem \ref{thm:classifying_interval-barcoding_invariants}). 
In addition, 
we show that, among homological invariants with respect to a set of indecomposables containing all the intervals and with finite relative global dimension, 
the interval Euler characteristic has the weakest discriminating power (cf. Theorem \ref{thm:hierarchy}).

We begin with a technical lemma;
the results therein follow largely from the existing literature.

\begin{lemma}
  \label{lem:fullconvex}
    Let $P$ be a finite poset, and let $P'\subset P$ be a convex full subposet. 
    \begin{enumerate}[label=(\roman*)]
    \item \label{lem:paddingsubcat}      
      The category $\pfdrep P'$ embeds as a full subcategory of $\pfdrep P$ through the inclusion functor
      $\iota : \pfdrep P' \hookrightarrow \pfdrep P$,
      which extends representations of $P'$ by assigning
      zero vector spaces and zero linear maps on $P\setminus P'$.
    \item \label{lem:paddingintervals}      
      The inclusion functor $\iota$ induces a natural inclusion
        $\iota': \ksp(\add \Int(P')) \hookrightarrow \ksp(\add \Int(P))$.
    \end{enumerate}
    In what follows, we consider the additive invariants
    $\pi^\Int,\ \rk^\Int,\ \beta^\Int_0$ and $\chi^\Int$
    on
    $\add \Int(P)$,
    and also  $\pi_{P'}^\Int,\ \rk_{P'}^\Int,\ \beta^\Int_{0,P'}$, and $\chi^\Int_{P'}$ on $\add \Int(P')$.
 \begin{enumerate}[label=(\roman*),resume]
 \item \label{lem:paddingintervalmult} For any representation $X$ of $P'$ we have   
     \begin{align}\label{eq:four equalities}
     \begin{split}
       \pi^\Int([\iota(X)]) & = \iota'\left( \pi_{P'}^\Int([X]) \right) \\
       \rk^\Int([\iota(X)]) & = \iota'\left( \rk^\Int_{P'}([X]) \right) \\
       \beta^\Int_{0}([\iota(X)]) & = \iota'\left( \beta^\Int_{0, P'}([X]) \right) \\ 
       \chi^\Int_{}([\iota(X)]) & = \iota'\left( \chi^\Int_{P'}([X]) \right).
      \end{split} 
     \end{align}     
 \end{enumerate} 
  \end{lemma}
\begin{proof}  %
  \ref{lem:paddingsubcat}: This is a standard result. See for example \cite{aoki2023summand}, in particular the paragraph immediately preceding Lemma 3.13 of \cite[Section 3.2]{aoki2023summand}.
    
  \noindent \ref{lem:paddingintervals}: This follows from \cite[Lemma 3.13]{aoki2023summand}.
  
  \noindent  \ref{lem:paddingintervalmult}: For the interval-multiplicity invariant $\pi^\Int$, the equality directly follows by viewing $\pfdrep P'$ as a full subcategory of $\pfdrep P$.

  Next, consider the generalized rank invariant $\rk^\Int$.  
  Let $I\in \Int(P)$.
  The coefficient of $[I]$ in $\rk^\Int([\iota(X)])$
  is equal to the multiplicity of $k_I$ as a direct summand of $\iota(X)|_I$  \cite[Lemma~3.1]{chambers2018persistent}
  \cite[Remark~6.16]{asashiba2024interval}.
  If $I \in \Int(P') \left(\subseteq \Int(P)\right)$, then we have $\iota(X)|_I = X|_I$, and thus  
  \[
    \rk^\Int([\iota(X)])(I) = \iota'\left( \rk_{P'}^\Int([X]) \right)(I).
  \]  
  We claim that the preceding equality holds even under the assumption that $I\not\in \Int(P')$.
  Suppose this is the case.  
  Then, the RHS is zero by definition.
  The LHS is also zero because
  by the definition of $X$, $\iota(X)|_I$ has support contained in $P'$, but 
  $k_I$ is not contained in $P'$, and thus
  $\iota(X)|_I$ cannot have $k_I$ as a direct summand.

  The last two equalities in Equation~\eqref{eq:four equalities} follow from \cite[Theorem~3.14~(1)~and~(2)]{aoki2023summand}. 
\end{proof}

In what follows, we show that certain pairs of interval-barcoding invariants are incomparable.

\begin{lemma}
  \label{lem:D4}
  For the posets $D_{1,2}$ and $D_{0,3}$ given by the following Hasse diagrams,
  the invariants $\beta_0^\Int$, $\chi^\Int$, $\dgm^\Int$ are pairwise unequal
  (and thus incomparable by Theorem~\ref{theorem:comparing}).
  \begin{center}  
  $D_{1,2}:
    \begin{tikzpicture}[baseline=4mm]      
      \node (a) at (0,1) {$1$};
      \node (b) at(-1,0) {$2$}; 
      \node (c) at(0,0) {$3$};
      \node (d) at(1,0) {$4$}; 
      \draw[->] (a)--(c); 
      \draw[->] (c)--(b); 
      \draw[->] (c)--(d);
    \end{tikzpicture}$
    \hspace{10mm}    
  $D_{0,3}:
    \begin{tikzpicture}[baseline=4mm]      
      \node (a) at (0,1) {$1$};
      \node (b) at(-1,0) {$2$}; 
      \node (c) at(0,0) {$3$};
      \node (d) at(1,0) {$4$}; 
      \draw[->] (c)--(a); 
      \draw[->] (c)--(b); 
      \draw[->] (c)--(d);
    \end{tikzpicture}
    $  
\end{center}
\end{lemma}
\begin{proof}  
  It is well-known that over the poset $D_{1,2}$ (respectively, $D_{0,3}$),
  up to isomorphism, there exists exactly one non-interval indecomposable representation $X_{1}$ (respectively, $X_2$) and that both $X_1$ and $X_2$ have the same dimension vector $\smat{ &1& \\1&2&1}$.\footnote{
      To see this, we note the following.
      First, $D_{1,2}$ and $D_{0,3}$ are both Dynkin quivers with underlying Dynkin graph $\mathbb{D}_4$.
      By Gabriel's Theorem for a Dynkin quiver $Q$ \cite{gabriel1972unzerlegbare},
      the mapping that sends representations to their dimension vectors
      provides a bijection between the set of isomorphism classes of indecomposable representations of $Q$
      and the set of positive roots of the quadratic form of $Q$.
      Furthermore, the quadratic form does not depend on the orientation of $Q$.
      A complete list of the positive roots for $\mathbb{D}_4$
      is given for example in \cite[Example~VII.4.15(b)]{assem2006elements}.      
      See \cite[Chapter~VII~Sections~3,~4,~5]{assem2006elements} for details.
    }
  Below, we give the values of the three invariants evaluated on $[X_1]$ (and $[X_2]$ respectively),
  expressing intervals by their dimension vectors. We have
  \[
    \begin{aligned}[t]
      \beta_0^\Int([X_1]) & = \left[\smat{ &0& \\0&1&1}\right] + \left[\smat{ &1& \\1&1&1}\right] + \left[\smat{ &0& \\1&1&0}\right], \\
      \chi^\Int([X_1]) & = \left[\smat{ &0& \\0&1&1}\right] + \left[\smat{ &1& \\1&1&1}\right] + \left[\smat{ &0& \\1&1&0}\right] - \left[\smat{ &0& \\1&1&1}\right], \\
      \dgm^\Int([X_1]) & = \left[\smat{ &1& \\0&1&1}\right] + \left[\smat{ &0& \\0&1&0}\right] + \left[\smat{ &1& \\1&1&0}\right] - \left[\smat{ &1& \\0&1&0}\right],
    \end{aligned}
  \]
  which are pairwise unequal, and
  \[
    \begin{aligned}[t]
      \beta_0^\Int([X_2]) & = \left[\smat{ &1& \\0&1&1}\right] + \left[\smat{ &0& \\1&1&1}\right] + \left[\smat{ &1& \\1&1&0}\right], \\
      \chi^\Int([X_2]) & = \left[\smat{ &1& \\0&1&1}\right] + \left[\smat{ &0& \\1&1&1}\right] + \left[\smat{ &1& \\1&1&0}\right] - \left[\smat{ &1& \\1&1&1}\right], \\
      \dgm^\Int([X_2]) & = \left[\smat{ &0& \\0&1&1}\right] + \left[\smat{ &1& \\0&1&0}\right] + \left[\smat{ &0& \\1&1&0}\right] - \left[\smat{ &0& \\0&1&0}\right],
    \end{aligned}
  \]
  which are pairwise unequal.  

  See \cite[Example~2.6(2)]{aoki2023summand} for   
  the details of computing
  $\beta_0^\Int$ and $\chi^\Int$ using the information of the Auslander-Reiten quiver
  (for a representation-finite algebra).
  On the other hand, we compute $\dgm^\Int$   
  by computing the generalized rank invariant and then computing its M\"obius inversion (cf.~Example~\ref{example:gpd}).  
\end{proof}

\begin{theorem}
  \label{thm:classifying_interval-barcoding_invariants}  
  For any finite poset $P$ for which  
  the poset $D_{1,2}$ or $D_{0,3}$ of Lemma~\ref{lem:D4}
  is embedded as a convex full subposet,
  the following pairs of invariants are incomparable.
  
  \begin{enumerate}[label=(\roman*)]
  \item $\dgm^\Int$ and $\chi^\Int$.
    \label{item:classifying_interval-barcoding_invariants-dgm_and_chi}
  \item $\dgm^\Int$ and $\beta_0^\Int$.
    \label{item:classifying_interval-barcoding_invariants-dgm_and_betti_zero}
  \end{enumerate}

  For any finite poset $P$ such that
  there exists at least one indecomposable representation of $P$
  that is not isomorphic to an interval representation,  
  the following pairs of invariants  are incomparable.
  \begin{enumerate}[label=(\roman*),resume]
  \item $\pi^{\Int}$ and $\dgm^\Int$.
      \label{item:classifying_interval-barcoding_invariants-dgm_and_interval_multiplicity}
      
  \item $\pi^\Int$ and $\chi^\Int$.
    \label{item:classifying_interval-barcoding_invariants-interval multiplicity_and_chi}
  \item $\pi^\Int$ and $\beta_0^\Int$.
    \label{item:classifying_interval-barcoding_invariants-interval_multiplicity_and_betti_zero}
  \item $\chi^\Int$ and $\beta_0^\Int$.
    \label{item:classifying_interval-barcoding_invariants-chi_and_betti_zero}
  \end{enumerate}  
\end{theorem}
See Figure \ref{fig:poset_of_interval-barcoding_invariants} for an illustration.
We emphasize that for any incomparable pair $(f,g)$ of invariants in the preceding theorem, by Theorem~\ref{theorem:isokernel}, there exists the transfer isomorphism $\ker f\to \ker g$.
\begin{figure}
  \centering
  \includegraphics[width=0.88\textwidth]{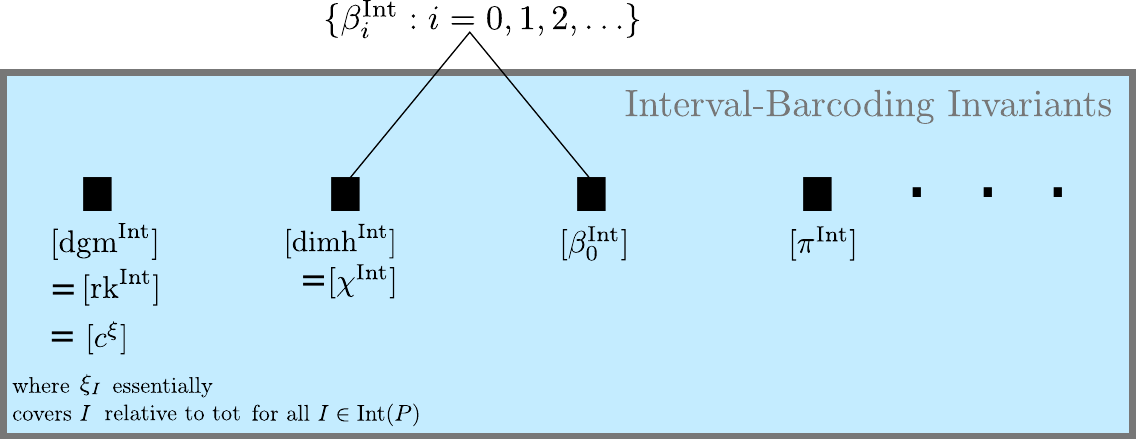}
  \caption{Illustration for Theorem \ref{thm:classifying_interval-barcoding_invariants},
    for any finite poset $P$ for which    
    the poset $D_{1,2}$ or $D_{0,3}$ of Lemma~\ref{lem:D4}
    is embedded as a convex full subposet.
    Inside the box is the Hasse diagram of the poset of equivalence classes of interval-barcoding invariants on $P$. The collection of interval Betti numbers $\{\beta_i^\Int\}_{i=0}^\infty$ is an invariant on $\fprepk P$ which determines both $\chi^\Int$ and $\beta_0^\Int$.
    The equivalence $\rk^\Int\sim\dgm^\Int$ follows from the definition of $\dgm^\Int$
    (cf. Example \ref{example:gpd}) and Corollary~\ref{cor:Mobius_does_not_affect_kernel}~\ref{item:Mobius_does_not_affect_kernel2}. 
    By Remark \ref{rem:essentially_covers},
    the equivalence class of $\rk^\Int\sim\dgm^\Int$
    also contains all compressed multiplicity invariants $c^\xi$
    with $\xi$ satisfying the condition that
    $\xi_I$ essentially covers $I$ relative to $\tot$ for all $I\in \Int(P)$,
    because under this condition $c^\xi = \rk^\Int$.}
    \label{fig:poset_of_interval-barcoding_invariants}
\end{figure}
\begin{proof}
  The four invariants ($\pi^{\Int}$, $\dgm^\Int$, $\chi^\Int$, and $\beta_0^\Int$)
  under consideration are interval-barcoding.
  Hence, by Theorem~\ref{theorem:comparing}, 
  in order to prove that any two of them
  are incomparable, it suffices to show that they are unequal.

  Let $P$ be any finite poset for which  
  $D = D_{1,2}$ or $D = D_{0,3}$ as defined in Lemma~\ref{lem:D4}
  is embedded as a convex full subposet.  
    

\noindent \ref{item:classifying_interval-barcoding_invariants-dgm_and_chi},
  \ref{item:classifying_interval-barcoding_invariants-dgm_and_betti_zero}:
By Lemma~\ref{lem:D4}, $\dgm_D^\Int$, $\chi_D^\Int$, $\beta_{0,D}^\Int$ are pairwise unequal,
  which implies that $\dgm^\Int$, $\chi^\Int$, $\beta_0^\Int$ are pairwise unequal by
  Lemma~\ref{lem:fullconvex}~\ref{lem:paddingintervalmult}.



Below, let $P$ be a finite poset such that there exists at least one indecomposable representation of $P$
that is not isomorphic to an interval representation.

\noindent
\ref{item:classifying_interval-barcoding_invariants-dgm_and_interval_multiplicity},
\ref{item:classifying_interval-barcoding_invariants-interval multiplicity_and_chi},
\ref{item:classifying_interval-barcoding_invariants-interval_multiplicity_and_betti_zero}:
Let $X\neq 0$ be any non-interval indecomposable representation of $P$.
Then, $\pi^\Int([X]) = 0$.
On the other hand, we will show that
$\dgm^\Int([X])$, $\chi^\Int([X])$, and $\beta^\Int_0([X])$ are nonzero.
First, suppose that $\dgm^\Int([X]) = 0$.
Then, since $\dgm^\Int$ is defined as the M\"obius inversion of $\rk^\Int$,
we have that $\rk^\Int([X]) = 0$, a contradiction.
Similarly, $\chi^\Int([X]) = 0$ implies that $\dimhom^\Int([X]) = 0$ by Proposition~\ref{proposition:dimhomandres}, a contradiction.
Finally, $\beta^\Int_0([X]) \neq 0$ follows from the fact that
an interval cover of $X$ provides a surjection from an interval decomposable module to $X\neq 0$,
and thus the interval cover must be nonzero.

\noindent \ref{item:classifying_interval-barcoding_invariants-chi_and_betti_zero}:
First, assume that there exists a representation $X$ of $P$ with interval dimension $1$,
and let $z := [X] - \beta^\Int_0([X])$. Then, $\beta^\Int_0(z) = 0$,
while 
\begin{align*}
    \res(z)
    & = \res([X]) - \res(\beta^\Int_0([X])) \\
    & = \sum_{i=0}^\infty (-1)^i \beta^\Int_i([X]) - \beta^\Int_0([X]) \\
    & = \sum_{i=1}^\infty (-1)^i \beta^\Int_i([X]) \\
    & = \beta^\Int_1([X]) \neq 0
\end{align*}
where the last equality and last inequality
both follow from the assumption that $X$ has interval dimension $1$.
Thus, under the assumption of the existence of such an $X$,
we have $\beta^\Int_0 \neq \res$, as needed.
Next, we show that there indeed exists such a representation $X$.
We note that the hypothesis on $P$ is equivalent
to assuming that $P$ has nonzero global interval dimension $d$,
which is finite by Proposition~\ref{prop:finite_interval_dimension}.
Then, there exists a representation $M$ with interval dimension equal to $d$,
and $\Omega^{d-1}(M)$ has interval dimension $1$
(recall the construction in Section~\ref{subsec:relhom_invariants}).
\end{proof}

Recall that the $m\times n$ \emph{commutative grid} is
the product of two totally ordered sets $\{1 < 2 < \hdots < m\} \times \{1 < 2 < \hdots < n\}$.  
The following is a specialization of Theorem~\ref{thm:classifying_interval-barcoding_invariants}
to the commutative grids.
\begin{corollary}
  \label{cor:commgrid}
  Let $P$ be the $m\times n$ commutative grid with $m \geq 2$ and $n \geq 3$
  (or symmetrically $m \geq 3$ and $n \geq 2$).  
  The interval-barcoding invariants
  \[
    \dgm^\Int, \chi^\Int, \beta_0^\Int, \pi^\Int
  \]
  are pairwise incomparable.  
\end{corollary}
\begin{proof}  
  The poset $P$ contains the poset $D_{1,2}$ as a convex full subposet, and 
  there exists a non-interval indecomposable representation of $P$. Hence, the conclusion follows from Theorem~\ref{thm:classifying_interval-barcoding_invariants}.
\end{proof}

\begin{remark}
  We remark that \cite[Corollary~7.10]{blanchette2024homological}
  is closely related to 
  Theorem~\ref{thm:classifying_interval-barcoding_invariants}~\ref{item:classifying_interval-barcoding_invariants-dgm_and_chi} (under the more restrictive hypothesis of having the $2\times 3$ commutative grid embedded as a convex full subposet),
  but the meaning of \textbf{embed} in \cite[Corollary~7.10]{blanchette2024homological} is not clear.
  Therefore, we presented above a proof of Theorem~\ref{thm:classifying_interval-barcoding_invariants}~\ref{item:classifying_interval-barcoding_invariants-dgm_and_chi}, which makes use of Lemma~\ref{lem:fullconvex}.
\end{remark}

Interestingly, the condition that
there exists a non-interval indecomposable representation of $P$
is not sufficient to guarantee that $\dgm^\Int$ and $\chi^\Int$ are incomparable.
The following lemma exhibits two such posets $P$.

\begin{lemma}
  \label{lem:D4equal}
  For the posets $D_{2,1}$ and $D_{3,0}$ given by the following Hasse diagrams, 
  the invariants $\dgm^\Int$ and $\chi^\Int$ are equal.
  \begin{center}  
    $D_{2,1}:
    \begin{tikzpicture}[baseline=4mm]      
      \node (a) at (0,1) {$1$};
      \node (b) at(-1,0) {$2$}; 
      \node (c) at(0,0) {$3$};
      \node (d) at(1,0) {$4$}; 
      \draw[->] (c)--(a); 
      \draw[->] (b)--(c); 
      \draw[->] (d)--(c);
    \end{tikzpicture}$
    \hspace{10mm}    
    $D_{3,0}:
    \begin{tikzpicture}[baseline=4mm]     
      \node (a) at (0,1) {$1$};
      \node (b) at(-1,0) {$2$}; 
      \node (c) at(0,0) {$3$};
      \node (d) at(1,0) {$4$}; 
      \draw[->] (a)--(c); 
      \draw[->] (b)--(c); 
      \draw[->] (d)--(c);
    \end{tikzpicture}
    $  
  \end{center}
\end{lemma}
Note that these posets are the opposite posets of the ones appearing in Lemma~\ref{lem:D4}.
\begin{proof}
Up to isomorphism, there exists exactly one non-interval indecomposable representation $X_{3}$ of $D_{2,1}$
(respectively, $X_4$ for $D_{3,0}$).\footnote{
    Note that $D_{2,1}$ and $D_{3,0}$ also have underlying graph the Dynkin graph $\mathbb{D}_4$.
    The same argument in the footnote concerning $D_{1,2}$ and $D_{0,3}$ in Lemma~\ref{lem:D4} 
    also holds here.
}
Both $X_3$ and $X_4$ have the same dimension vector $\smat{ &1& \\1&2&1}$.
To see the equality of the invariants,
since both invariants are interval-barcoding, it suffices to check the equality on $X_3$ and on $X_4$ respectively.
A direct computation shows that
\[    
  \chi^\Int([X_3])  = \left[\smat{ &1& \\0&1&1}\right] + \left[\smat{ &0& \\0&1&0}\right] + \left[\smat{ &1& \\1&1&0}\right] - \left[\smat{ &1& \\0&1&0}\right] = 
  \dgm^\Int([X_3])    
\]
and
\[
  \chi^\Int([X_4]) = \left[\smat{ &0& \\0&1&1}\right] + \left[\smat{ &1& \\0&1&0}\right] + \left[\smat{ &0& \\1&1&0}\right] - \left[\smat{ &0& \\0&1&0}\right] = 
  \dgm^\Int([X_4]).
\]
\end{proof}

Theorem~\ref{thm:classifying_interval-barcoding_invariants}
and Lemma~\ref{lem:D4equal} lead to the following question.
\begin{question}
  Characterize all finite posets $P$ over which the four invariants
  $\dgm^\Int, \chi^\Int, \beta_0^\Int, \pi^\Int$
  are pairwise incomparable.
\end{question}

For the incomparable invariants given in Corollary~\ref{cor:commgrid}, the following example demonstrates their incomparability through specific pairs of representations.
\begin{example}
  \label{example:dgm_and_dimhom}
  Consider the $k$-representation $M$ of the $2\times 3$ commutative grid $G$ given as:
  \[
    M:
    \begin{tikzcd}[ampersand replacement=\&]
      \bfield \arrow[r,"1"] \&
      \bfield \arrow[r] \&
      0
      \\
      \bfield \arrow[u,"1"]\arrow[r,"\smat{1\\1}"] \&
      \bfield^2 \arrow[r,"\smat{0 & 1}"] \arrow[u,"\smat{ 1 & 0}"] \&
      \bfield\arrow[u,"1"]
    \end{tikzcd}
  \]
  whose dimension vector is $\smat{1 & 1 & 0 \\ 1 & 2 & 1}$.
  Also, consider the representations of $G$
  \[
    N := M \oplus \smat{1 & 0 & 0 \\ 1 & 1 & 0}
    \text{ and }
    L :=
    \smat{1 & 0 & 0 \\ 1 & 1 & 1} \oplus
    \smat{1 & 1 & 0 \\ 1 & 1 & 0} \oplus
    \smat{0 & 0 & 0 \\ 0 & 1 & 0}
  \]
  where the dimension vectors stand for the corresponding interval representations.
  \begin{enumerate}
  \item The generalized persistence diagram   
    $\dgm^{\Int}$ 
    cannot distinguish $N$ and $L$,
    whereas the dim-hom invariant    
    $\dimhom^\Int$
    can, as noted in \cite[Proposition~7.8]{blanchette2024homological}.
    Namely,  
    \begin{equation}\label{eq:dgm_coincidence}
      \dgm^\Int([N]) = [L] = \dgm^\Int([L]),
    \end{equation}whereas, 
    for the interval $I = \smat{1 & 0 & 0 \\ 1 & 1 & 0}$,
    \begin{equation}\label{eq:dimhom_noncoincidence}
      \dimhom^\Int([N])(I) = 1 \neq 0 = \dimhom^\Int([L])(I).    
    \end{equation}   
    Thus by Corollary~\ref{corollary:comparingbarcodelike}, $\dgm^\Int$ and $\dimhom^\Int$ are incomparable. Then, since $\dimhom^\Int\sim \chi^{\Int}$ (cf. Proposition~\ref{proposition:dimhomandres}),  $\dgm^\Int$ and $\chi^\Int$ are incomparable.
    
  \item It is also clear that the interval multiplicity $\pi^\Int$ can distinguish $N$ and $L$. For example, 
  \[
    \pi^\Int([N])(I) = 1 \neq 0 = \pi^\Int([L])(I).    
  \]
  Thus, $\dgm^\Int$ and $\pi^\Int$ are incomparable by Theorem~\ref{theorem:comparing}.
\end{enumerate}
\end{example}

Next, we illustrate the use of the
transfer isomorphism
to generate a pair of representations that can be distinguished by one invariant, $f$, but not by another, $g$, starting from a pair that is distinguished by $g$ but not by $f$.
\begin{example}\label{ex:use_of_flipping_map}
  Consider $N$ and $L$ in the previous example. Let $x:= [N]-[L] \in \ksp(\fprepk P)$.
  By Equations~(\ref{eq:dgm_coincidence})~and~(\ref{eq:dimhom_noncoincidence}),
  and the fact that $\dimhom^\Int \sim \chi^\Int$, we have
  \[
    x \in \ker(\dgm^\Int) \text{ and } x \not \in \ker(\res).
  \]  
  As we showed in the proof of Theorem~\ref{theorem:isokernel},
  for $y := x - \res(x)$,  
  we have
  \[
    y \not\in \ker(\dgm^\Int) \text{ and } y \in \ker(\res).
  \]
  By Lemma \ref{lemma:invariant_kernel}~\ref{lemma:invariantrelations:num:kerinterpret},
  the positive part $Y_+$ and the negative part $Y_-$ of $y$ can be distinguished by $\dgm^\Int$, but not by $\res$. While we know this without explicitly computing $y$, let us compute $y$ nonetheless. First, note that $L$ is interval decomposable and thus $\res([L])=[L]$. Also, for the representation $M$ from the previous example, a direct computation gives
  the interval resolution of $M$ as:
  \[
    0\longrightarrow \smat{1 & 1 & 0\\ 0& 1 & 1}\longrightarrow \smat{0 & 0 & 0 \\ 0 & 1 & 1}
    \oplus \smat{1 & 1 & 0\\ 0& 1 & 0}
    \oplus \smat{1 & 1 & 0\\ 1 & 1 & 1} \stackrel{f}{\longrightarrow} M \longrightarrow 0,
  \]
  implying that $\res(M)=
  \left[\smat{0 & 0 & 0 \\ 0 & 1 & 1}\right]
  + \left[\smat{1 & 1 & 0 \\ 0 & 1 & 0}\right]
  + \left[\smat{1 & 1 & 0 \\ 1 & 1 & 1}\right]-\left[\smat{1 & 1 & 0\\ 0& 1 & 1}\right]$.
  
  Since $N=M\oplus \smat{1 & 0 & 0 \\ 1 & 1 & 0}$, we have $x=[N]-[L]=[M]+\left[\smat{1 & 0 & 0 \\ 1 & 1 & 0}\right]-[L]$, and thus additivity of $\res$ implies:
  \begin{align*}
    y
    & = x - \res(x) \\
    & = \left([M] + \left[\smat{1 & 0 & 0 \\ 1 & 1 & 0}\right] - [L]\right)
                                                     - \left(\res([M]) + \res\left(\left[\smat{1 & 0 & 0 \\ 1 & 1 & 0}\right]\right) - \res([L])\right) \\
    & = [M] - \res([M]) 
    \\
    & = [M] + \left[\smat{1 & 1 & 0\\ 0& 1 & 1}\right]
                                             - \left[\smat{0 & 0 & 0 \\ 0 & 1 & 1}\right]
                                                                                - \left[\smat{1 & 1 & 0 \\ 0 & 1 & 0}\right]
                                                                                                                   - \left[\smat{1 & 1 & 0 \\ 1 & 1 & 1}\right],
  \end{align*} 
  where the third equality follows from $\res([L])=[L]$ and $\res\left(\left[\smat{1 & 0 & 0 \\ 1 & 1 & 0}\right]\right)=\left[\smat{1 & 0 & 0 \\ 1 & 1 & 0}\right]$.
  Hence,
  \[
    Y_+ = M \oplus \smat{1 & 1 & 0\\ 0& 1 & 1}
    \text{ and }
    Y_- = \smat{0 & 0 & 0 \\ 0 & 1 & 1}
    \oplus \smat{1 & 1 & 0\\ 0& 1 & 0}
    \oplus \smat{1 & 1 & 0\\ 1 & 1 & 1}
  \]
  are not distinguishable by $\res$.
\end{example}

Next we show that, among homological invariants 
with respect to a set of indecomposables containing all the intervals
and with finite relative global dimension,
the interval Euler characteristic has the weakest discriminating power.

\begin{theorem}[Hierarchy of homological invariants]\label{thm:hierarchy}
  Let $P$ be a finite poset and let $\cD:= \pfdrep P$. Let $\Int(P)\subseteq Q_1\subsetneq Q_2\subseteq \ind(\cD)$ and  $\cC_1:=\add Q_1$, and $\cC_2:=\add Q_2$ such that the global $\cC_1$-dimension and global $\cC_2$-dimension are both finite.
  Then, any homological invariant relative to $Q_2$ is strictly finer than any homological invariant relative to $Q_1$.
\end{theorem}
\begin{proof}
  Let $f$ be a homological invariant relative to $Q_1$,
  and let $g$ be a homological invariant relative to $Q_2$.
  By definition, $f \sim \chi^{Q_1}$ which means that $f$ is a $\cC_1$-barcoding-equivalent invariant.
  Likewise, $g \sim \chi^{Q_2}$ is a $\cC_2$-barcoding-equivalent invariant.
  We then have
  \begin{align*}
    g&\sim \dimhom_{\cD}^{Q_2}&\mbox{by \cite[Theorem~4.22]{blanchette2024homological}}
    \\&\gtrsim \dimhom_{\cD}^{Q_1}&\mbox{by definition}
    \\&\sim f &\mbox{by \cite[Theorem~4.22]{blanchette2024homological}}
  \end{align*}
  By Corollary \ref{cor:comparison}, $g\gtrsim f$ implies that $g\gnsim f$, completing the proof.
\end{proof}

\subsection{Applications to interval-barcoding-equivalent invariants
  for representations of infinite posets}

In the previous section, we dealt with barcoding-equivalent invariants for representations of \emph{finite} posets.
We now turn to barcoding-equivalent invariants for representations of
\emph{infinite} posets $P$.

First, we consider the setting 
$P=\R^2$
and $\cD = \fprepk P$.
Here,
$\dgm^\Int$ is defined \cite[Definition~3.1 and Theorem~C (iii)]{clause2022discriminating}, 
finitely supported \cite[Proposition~3.17]{clause2022discriminating},
additive (see the proof of Theorem \ref{thm:uniqueness} below),
fixes intervals (Theorem \ref{thm:dgm_generalize_the_barcode_and_completeness}~\ref{item:dgm_generalize_the_barcode}),  
and is thus a barcoding-invariant (cf. Remark \ref{footnote}).

\begin{remark}
  By definition, any $\cC$-barcoding invariant
  $f: \ksp(\cD) \rightarrow \ksp(\cC)$
  has codomain $\ksp(\cC) =\Z^{(\ind(C))}$.
  This implies that the value $f([M])$, viewed as a function on $\ind(\cC)$,  
  must be finitely supported for each $M \in \cD$.
  This is not necessarily the case for barcoding-\emph{equivalent} invariants.
  For example, when $P=\R^2$ and  $\cD = \fprepk P$,
  the generalized rank invariant (Example~\ref{ex:GRI})
  is not finitely supported, but is equivalent to 
  $\dgm^\Int$ which is is finitely supported and barcoding.  
\end{remark}

Next, we obtain a new characterization of $\dgm^\Int$ that does not invoke M\"obius inversion. 

\begin{theorem}
[Universality of the generalized persistence diagram]\label{thm:uniqueness} 
Let $P$ be either $\R$, $\R^2$, or a finite poset, and $\cD:=\fprepk P$.
The invariant $\dgm^\Int$ is the \emph{unique} additive invariant on $\cD$ that is equivalent to the generalized rank invariant $\rk^\Int$ and fixes intervals.
\end{theorem}

\begin{proof}
  Firstly, we prove that, in either setting, $\dgm^\Int$ is additive.
  Let $M,N$ be any two $P$-representations where
  $\dgm^\Int_M:=\dgm^\Int([M])$ and $\dgm_N^\Int:=\dgm^\Int([N])$ exist.
  Then, for all $I\in \Int(P)$, 
  \begin{align*}
    \rk^\Int([{M\oplus N}])(I) &=\rk^\Int([M])(I)+\rk^\Int([N])(I) \\
                               &=\sum_{\substack{J\supseteq I\\J\in \Int(P)}} \dgm_M^\Int(J)+\sum_{\substack{J\supseteq I\\J\in \Int(P)}} \dgm_N^\Int(J)&\mbox{} \\
                               &=\sum_{\substack{J\supseteq I\\J\in \Int(P)}}(\dgm^\Int_M+\dgm^\Int_N)(J)
  \end{align*}
  where the first equality follows from the additivity of $\rk^\Int$,
  and the second equality from \cite[Definition~3.1~and~Theorem A]{clause2022discriminating}.
  %
  %
  Again by 
  \cite[Definition~3.1~and~Theorem A]{clause2022discriminating},
  we conclude that $\dgm^\Int([M\oplus N])$ exists and coincides with $\dgm_M+\dgm_N$.

  By Theorem \ref{thm:dgm_generalize_the_barcode_and_completeness}~\ref{item:dgm_generalize_the_barcode}, $\dgm^\Int$ fixes intervals.
  Let $f$ be any additive invariant that is
  equivalent to the generalized rank invariant and fixes intervals.
  Then, since $\dgm^\Int\sim \rk^\Int \sim f$,
  we have that $\dgm^\Int \sim f$.
  By Theorem \ref{theorem:comparing}, we have $\dgm^\Int=f$. 
\end{proof}

Let the category $\cD$ be either $\fprepk P$ or $\pfdfdrep P$.
Then, the multiplicity invariant restricted to any set of indecomposables $Q$
(Example~\ref{ex:restricted_multiplicity})
satisfies the finite support condition (since $\cD$ is Krull-Schmidt), is additive, 
and is a barcoding-invariant.

Hence, the previous theorem directly implies:
 
\begin{corollary}
  Let $P$ be either $\R^1$, $\R^2$, or a finite poset, and $\cD:=\fprepk P$. Then, the barcoding invariants $\dgm^\Int$ and $\pi^\Int$  on $\cD$ are incomparable.
\end{corollary}


\section{Discussion}
\label{sec:discussion}

There exist numerous barcoding-equivalent invariants in the literature
(cf. Example~\ref{example:compressionbasedinvariants}).
We further emphasize 
that any choice of compression system $\xi$ yields a compression multiplicity $c^\xi$, which is an interval-barcoding-equivalent invariant (cf. Example~\ref{example:compressionbasedinvariants}~\ref{ex:cm_be}).
 

Theorem~\ref{theorem:comparing} implies that any pair of barcoding invariants are either identical or incomparable, indicating that comparing barcoding invariants with the partial order $\gtrsim$ (which compares their kernels) is too rigid.
Also, our results provide a new characterization of the interval Euler characteristic and that of the generalized persistence diagram  (cf.~Theorems~\ref{thm:hierarchy} and \ref{thm:uniqueness}).

The existence of the \emph{transfer} isomorphism between the kernels of barcoding invariants, established in Theorem~\ref{theorem:isokernel}, provides a direct link between these kernels
beyond that provided by Theorem~\ref{theorem:comparing}.
Namely, while Theorem~\ref{theorem:comparing} allows for a simple determination of whether two interval barcoding invariants $f$ and $g$ are incomparable---specifically by finding a single pair $(M,N)$ of representations distinguished by one invariant $f$ but not by another $g$,
Theorem~\ref{theorem:isokernel} goes further by enabling us to utilize that same pair $(M,N)$ to construct a pair of representations distinguished by $g$ but not by $f$.
The utility of combining Theorems~\ref{theorem:comparing} and~\ref{theorem:isokernel} is demonstrated in Theorem~\ref{thm:classifying_interval-barcoding_invariants}, Corollary~\ref{cor:commgrid}, and Examples~\ref{example:dgm_and_dimhom} and~\ref{ex:use_of_flipping_map}.
 
We emphasize that the bijection between $\ker f \setminus \ker g$ and $\ker g \setminus \ker f$ given in Theorem~\ref{theorem:isokernel} implies that no interval-barcoding invariant can be fundamentally better or worse than the others, and this bijectivity is not expected of arbitrary isomorphisms
(cf. Remarks~\ref{rem:special properties of T}, \ref{rem:depiction} and
\ref{rem:generic iso is too weak}).

Theorems~\ref{theorem:comparing} and \ref{theorem:isokernel} suggest a need to develop
data-dependent, probabilistic, and/or statistical approaches
to analyzing the discriminating power of barcoding invariants.
We also note that, with knowledge of the specific data being analyzed using poset representations,
certain (interval-)barcoding invariants may possess superior discriminating power,
potentially leading to conclusions that significantly differ from ours (cf.~Remark ~\ref{rem:depiction}).

While the comparison framework of this work does not encompass non-additive invariants,
recent progress (e.g., \cite{bauer2025multi,bjerkevik2025stabilizing})
suggests that respecting the direct sum structure may be too restrictive from the point of view of stability,
  and thus 
  {the development and comparison of } non-additive invariants would also be valuable.

\section*{Acknowledgments}
E.G.E.\ would like to thank
  Hideto Asashiba for discussions on the interval replacements, and
  Thomas Brüstle for discussions on the comparison of invariants.
  E.G.E.\ would like to thank Bjørnar Hem for suggesting 
  a weakening of the assumptions in Proposition~\ref{proposition:applyequivalent},
  which leads to the current statement.

  W.K. would like to thank Donghan Kim for insightful conversations,
  and Wojciech Chach\'olski for proposing   
  an alternative construction  of an isomorphism between kernels of barcoding invariants.
  
Both authors would also like to thank the attendees of the 
BIRS workshop ``Representation Theory and Topological Data Analysis (24w5241)''\footnote{Work on this project started prior to both authors attending the BIRS workshop.} for stimulating discussions on multi-parameter persistence and representation theory.
The authors also thank the anonymous reviewers for their thoughtful comments, which prompted the inclusion of Remark~\ref{rem:generic iso is too weak},
and prompted us to revise the presentation of main results in this paper, primarily in Section~\ref{subsec:transfer iso}.

E.G.E.\ is supported by JSPS Grant-in-Aid for Transformative Research Areas (A) (22H05105). WK is  supported by the National Research Foundation of Korea(NRF) grant funded by the Korea government(MSIT) (RS-2025-00515946).



\bibliographystyle{plain}
\bibliography{refs}

\end{document}